\documentclass[reqno,10pt]{amsart}
\usepackage[english]{babel}
\usepackage{amsthm}
\usepackage{amsmath}
\usepackage{amssymb}
\usepackage[latin1]{inputenc}
\usepackage[normalem]{ulem}% overlines
\usepackage{tikz}
\usepackage{tikz-3dplot}
\usepackage{pgfplots}
\usepackage{enumerate}

\usepackage[all]{xy}
\usepackage{dsfont}
\usepackage[bookmarksnumbered,colorlinks]{hyperref}
\usepackage{graphics}
\usepackage{graphicx}
\usepackage{enumerate}

\def\bb#1\eb{\textcolor{blue}
{#1}} %
\def\br#1\er{\textcolor{red}
{#1}} %
\def\bbr#1\ebr{\textcolor{brown}
{#1}} %
\newcommand{\df}{{\rm d}}

\newcommand{\R}{\mathds R}
\newcommand{\N}{\mathds N}
\newcommand{\Z}{\mathds Z}

\newcommand{\LL}{\mathds L}

\newcommand{\be}{\begin{equation}}
\newcommand{\ee}{\end{equation}}

   \def\br#1\er{\textcolor{red}{#1}} %
      \def\bb#1\eb{\textcolor{blue}{#1}} %

\title[Some criteria for WRS completeness and Cauchy hyp.]{Some criteria for Wind Riemannian completeness \\ and existence of Cauchy hypersurfaces}

\author[M. A. Javaloyes]{Miguel Angel Javaloyes}
\address{Departamento de 
Matem\'aticas, \hfill\break\indent
Universidad de Murcia, \hfill\break\indent
Campus de Espinardo,\hfill\break\indent
30100 Espinardo, Murcia, Spain}
\email{majava@um.es}

\author[M. S\'anchez]{Miguel S\'anchez}
\address{Departamento de Geometr\'{\i}a y Topolog\'{\i}a, Facultad de Ciencias, \hfill\break\indent
 Universidad de Granada,\hfill\break\indent
 Campus Fuentenueva s/n,
 \hfill\break\indent 18071 Granada, Spain}
\email{sanchezm@ugr.es}

\date{31.01.2017}

\thanks{2010 {\em Mathematics Subject Classification:} Primary  53C60, 53C22 \\
\textbf{Key words:} Globally hyperbolic spacetime, Killing vector field, stationary and SSTK spacetimes, Finsler metrics, Randers and
Kropina metrics, Zermelo navigation, wind Finslerian structure.}

\begin{document}
\newtheorem{thm}{Theorem}[section]
\newtheorem{prop}[thm]{Proposition}
\newtheorem{lemma}[thm]{Lemma}
\newtheorem{cor}[thm]{Corollary}
\theoremstyle{definition}
\newtheorem{defi}[thm]{Definition}
\newtheorem{notation}[thm]{Notation}
\newtheorem{exe}[thm]{Example}
\newtheorem{conj}[thm]{Conjecture}
\newtheorem{prob}[thm]{Problem}
\newtheorem{rem}[thm]{Remark}
\newtheorem{conv}[thm]{Convention}
\newtheorem{crit}[thm]{Criterion}

\begin{abstract}
Recently, a link between Lorentzian and Finslerian Geometries has been  carried out, leading to the notion of {\em wind Riemannian structure} (WRS), a generalization of Finslerian Randers metrics. Here, we further develop  this notion and its applications to spacetimes, by introducing some characterizations and criteria for the completeness of  WRS's. 

As an application, we consider a general class of spacetimes admitting a time function $t$ generated by the flow of a complete Killing vector field  
(generalized standard stationary spacetimes or, more precisely, SSTK  ones)  and derive simple criteria ensuring that its slices $t=$ constant are Cauchy. Moreover, a brief summary on the Finsler/Lorentz link  for readers with some acquaintance in Lorentzian Geometry,  plus some simple examples  in Mathematical Relativity, are provided. 
\end{abstract}

\maketitle

%\newpage
\tableofcontents
%\newpage

\section{Introduction}

In the last years, a fruitful link between   Lorentzian and  Finslerian geometries has been refined  
more and more;  indeed, ramifications to different areas such as control theory, have also  appeared. In this framework, our purpose here is twofold. First,  a brief summary on the subject for an audience of Lorentzian geometers is provided. Then, a new application   
for Lorentzian Geometry will be obtained. Namely, we will consider a big class of spacetimes admitting a time function $t$ (generated by the flow of a complete Killing  vector field $K=\partial_t$), and we will characterize when its slices $t=$ constant are Cauchy hypersurfaces, providing also simple criteria that ensure this property.  

The refinements of the Lorentz/Finsler link can be understood in three steps:

\begin{enumerate}
\item Firstly, consider  a product  $\R\times M$ endowed with a standard Lorentzian product metric $g=-dt^2+g_0$ (see below for more precise notation and definitions). Clearly, all the properties of the spacetime will be encoded in the Riemannian metric $g_0$. Even more, if one considers a standard static metric $g=-\Lambda dt^2+g_0$ ($\Lambda>0$), then all the conformal properties of the spacetime can be studied in the conformal representative $g/\Lambda$ and, thus, they are encoded in the Riemannian metric $g_R=g_0/\Lambda$ (see for example \cite[Theorem 3.67]{BEE} or \cite{Har} and references therein).
\item Secondly,  the previous case can be generalized by allowing $t$-independent cross-terms, i.e. the stationary metric $g=-\Lambda dt^2+2\omega dt + g_0$ ($\Lambda>0$, $\omega$ 1-form on $M$). In this case, the conformal properties are provided by a Randers metric (formula \eqref{e_Randers} below), which is a special class of Finsler metric characterized by a Riemannian metric $g_R$ and a vector field $W$ with $|W|_R<1$  (according to the interpretation of Zermelo problem, \cite[Prop. 1.1]{BRS}).  This correspondence is carried out in full detail in \cite{CJS} and many related properties can be seen in \cite{CGS, CaJaMa,  GHWW, FHS, JLP15}. 
\item Finally, suppress the restriction $\Lambda>0$ in the previous case (SSTK splitting).  Now, the conformal structure is still characterized  by a pair $(g_R,W)$, but the restriction $|W|_R<1$ does not apply. From the Finslerian viewpoint, this yields a {\em wind Riemannian structure} (WRS), which is a generalization of Randers metrics; even more, this also suggests the subsequent generalization  of all Finsler metrics: {\em wind Finslerian structures}. 
Such new structures were introduced and extensively studied in \cite{CJSwind} and further developments and applications for Finslerian geometry have been carried out in \cite{JS, JV}.     
\end{enumerate}
The study of SSTK splittings in \cite{CJSwind} includes quite a few  topics. Among them, relativists may be interested in a link between the well-known (and non-relativistic) problem of  {\em Zermelo navigation} and the relativistic {\em Fermat principle}; indeed, this link allows one to solve both problems beyond their classical scopes; moreover, it shows connections with the so called {\em analogue gravity} \cite{BLV}. On the other hand, 
global properties of causality of SSTK spacetimes are neatly characterized by their Finslerian counterparts. So, the exact step in the ladder of causality of SSTK spacetimes is described sharply in terms of the associated WRS. In particular, the fact that the slices $t=$ constant  are Cauchy hypersurfaces becomes equivalent to the  (geodesic) completeness of the WRS.

Even though such results are very accurate, a difficulty appears from a practical viewpoint. WRS's  are not standard known elements, as they have been introduced only recently. Therefore,  to determine whether  they satisfy or not some geometrical properties  may be  laborious.   
 Due to this reason, our purpose here is to introduce some simple notions and results which allow one to check easily properties of WRS's. As a first approach, we will focus on results about completeness because,  on the one hand, completeness has a direct translation to spacetimes in terms of Cauchy hypersurfaces and, on the other, it is a basic natural property with applications to Finslerian Geometry (see \cite{JS}). However, the introduced tools are expected to be applicable for other properties too, and some hints are made in the examples at the end.
 
 Our task is organized as follows. In Section \ref{s_2} we give an overview on the Lorentz/Finsler correspondence for readers with some knowledge on Lorentzian Geometry (compare with the overview in   \cite{JS}, written for a more Finslerian audience). More precisely, in subsection \ref{s_2.1} we describe the class of spacetimes to be studied (generalized standard stationary  or, more precisely, SSTK spacetimes).   The generality of this class, which includes many typical relativistic spacetimes, is stressed, and the way to obtain a (non-unique) {\em SSTK splitting} is detailed. Then, in subsection \ref{s_2.2}, the relation between the conformal  classes of spacetimes 
and the properties of associated Finslerian structures is introduced gradually, with increasing generality in the Finslerian tools: Riemannian/Randers/Randers-Kropina/WRS. In subsection \ref{s_2.3}, as a toy application of the correspondence,  we consider a causally-surprising example of static spacetime constructed recently by Harris \cite{Har}, and we show its  Finslerian counterpart, explaining the corresponding curious properties which appear in the distance of the  associated Finsler manifold.

Section \ref{s_3} makes both, to introduce a geometric element for the practical study of WRS's and to prove our main result on completeness. subsection \ref{s_3.1} explains  the      
precise technical notions on WRS balls, geodesics and completeness, extracted from \cite{CJSwind}. Then, in subsection \ref{s_3.2} we introduce a new key ingredient,  the {\em extended conic Finsler metric} $\bar F$ associated with any WRS. We emphasize 
that,   as analyzed in \cite{CJSwind}, any WRS  $\Sigma$ determines both,  a {\em conic Finsler}  metric $F$ and a {\em Lorentz-Finsler} one $F_l$. The former differs from a 
standard Finsler metric only in the fact that its domain is just an open conic region of the tangent bundle (this is a possibility with independent interest, see
 \cite{JSpisa}). However, our aim here is to show that, for any such $\Sigma$, the conic metric  $F$ admits a natural extension $\bar F$ to 
 the boundary of the conic region; moreover, $\bar F$ has  an associated exponential, distance-type function (called here {\em $\bar F$-separation}), Cauchy sequences, etc.   Then, in subsection \ref{s_3.3}. 
 we give our main result, Theorem \ref{HopfRinow}, which contains a double goal: to 
 prove that the completeness of the WRS $\Sigma$ is fully  equivalent to the 
 completeness of $\bar F$, and to  show that $\bar F$ satisfies a set of properties in the spirit of Hopf-Rinow Theorem. These properties will allow us to determine if $\bar F$ and, then, $\Sigma$, are complete. 

In Section \ref{s_4} we derive some  applications by using the previous Theorem \ref{HopfRinow}. In subsection \ref{s_4.1}, some  simple criteria for checking whether a WRS is complete or not are provided. These criteria  are  stated in both, natural WRS elements and the original SSTK metric. So, one obtains also  criteria which ensure whether the slices of an SSTK spacetime are Cauchy, in an easily manageable way. Finally, subsection \ref{s_4.2} ends with some further concrete examples and prospects  in Mathematical Relativity.

\section{A Lorentzian overview on Wind Riemannian Structures}\label{s_2}

\subsection{Revisiting SSTK spacetimes}\label{s_2.1}

We will follow standard conventions and background results  as in \cite{BEE, MinSan}. In particular, 
a spacetime
$(L,g)$ is  a time-oriented connected Lorentzian manifold $(-,+,\dots , +)$ of dimension $n  +1,  n  \geq 1$.  Lightlike vectors $v\in TL$ will satisfy both, $g(v,v)=0$ and $v\neq 0$ (while  null vectors would be allowed to be equal to 0) so, causal vectors, being either timelike or lightlike, also exclude 0.  Except when otherwise specified,   $(L,g)$ will  also be  stably causal, so that it admits a {\em temporal function} $t: L\rightarrow \R$ according to  \cite{BS05,SaBrasil} (that is, $t$ is smooth and onto,  with timelike past-directed gradient $\nabla t$ and, in particular, a time function). 
     Let us start with a simple result when a vector field $K\in \mathfrak{X}(L)$ complete and transversal to the slices of $t$ can be chosen.
%such that $dt(K)\equiv 1$. In the case that  $K$ can be chosen complete, the spacetime splits as follows. 

\begin{prop}\label{p_2.1}
Let $t$ be a temporal function for $(L,g)$ and $K\in \mathfrak{X}(L)$ complete  with flow $\varphi:\R\times L\rightarrow L$ such that $\df t(K)\equiv 1$. Putting $M:=t^{-1}(0)$, the map
$$
\Phi:\R\times M \rightarrow L, \qquad (\bar t,x)\mapsto \varphi_{\bar t}(x) 
$$ 
is a diffemorphism such that: (a) $t\circ \Phi:\R\times M\rightarrow \R$ agrees with the natural projection and (b) $t\circ \Phi$ is also a temporal function for the pull-back metric $\Phi^*(g)$ and time orientation  induced on $\R\times M$ via $\Phi$ (which %, $t\circ \Phi$ 
will be denoted simply as $t$). Then, $\Phi_*(\partial_t)=K$ and %the   on $\R\times M$
%, also denoted $g$ in what follows,  
%can be written:
\begin{equation}\label{lorentz0}
\Phi^*(g)=- \Lambda^L \df t^2+\omega^L\otimes \df t+\df
t\otimes \omega^L+ g_0^L,
\end{equation}
where $\Lambda^L$, $\omega^L$ and $g_0^L$ are, respectively, the smooth
real function $g(K,K)\circ \Phi$, a one form whose kernel includes $\partial_t$ and a positive semi-definite metric tensor whose radical is spanned by $\partial_t$, all of them defined on $\R\times M$. 

\end{prop}

\begin{proof}  $\Phi$ is onto because  $z=\Phi(t(z), \varphi_{-t(z)}(z))$ for all  $z\in L$, and one-to-one because the equality $\df t(K)\equiv 1$ forbids   the unique integral curve of $K$ through $z$ to close. % cross $M$ twice. 
That equality also implies $\Phi (\{t_0\}\times M) = t^{-1}(t_0)$ plus the transversality of $K$ and the slices $t^{-1}(t_0)$, so that $\Phi$  becomes a (local) diffeomorphism, and all the other assertions follow easily.  
\end{proof}
As already done for the natural projection $\R\times M\rightarrow \R$, the diffeomorphism $\Phi$ will be omitted with no further mention  in the remainder.
\begin{rem} 
(1) As any timelike vector field $T$ satisfies that $\df t(T)$ cannot vanish, the normalized vector $K=T/\df t(T)$ satisfies $\df t(K)\equiv 1$; in particular, one can choose $K=\nabla t/g(\nabla t, \nabla t)$.
However, Proposition \ref{p_2.1} shows that the completeness of $K$ is much more difficult to obtain, even taking into account that, in general, $K$ is not assumed to be  timelike.  

Indeed, no complete $K$ can exist in  a stably causal spacetime $(L,g)$ such that $L$ is not a smooth product manifold $\R\times M$  as, for example,  the spacetime obtained by removing two points from Lorentz-Minkowski spacetime $\LL^2$. However, the completeness of $K$ may not be achieved even when $L$ is a product. Indeed, this is the case of  $\LL^2\setminus\{0\}$: this spacetime  is diffeomorphic to $\R\times S^1$ and  its natural coordinate $t=x^0$ is a temporal function; nevertheless, it contains the non-homeomorphic slices $t=0$ and $t=1$. 

(2) The choice $K= \partial_t+2\partial _x$ in the globally hyperbolic strip $$L=\{(t,x)\in \LL^2: x=2t+\lambda, \forall \lambda \in (-1,1),\forall t\in\R\}$$  shows that a complete choice of $K$ may be possible even when no complete timelike choice exists. However, in a globally hyperbolic spacetime one can always choose a   temporal function $t$ whose levels are Cauchy hypersurfaces \cite{BS05}. In this case, the choice $K=\nabla t/g(\nabla t, \nabla t)$ suggested above is necessarily complete (notice that temporal functions are assumed to be onto, and the integral curves of  $K$ must cross all the slices of $t$, due to its Cauchy character).
As a last observation, notice that the onto character of a temporal function can be deduced when a complete $K$ satisfying  $\df t(K)\equiv 1$ exists. 

(3) A straightforward computation  shows that a metric on $\R\times M$ written as in \eqref{lorentz0} is Lorentzian if and only if for each  tangent vector $u$ to $\R\times M$ such that $g^L_0(u,u)\neq 0$, 
\begin{equation}
\label{lorentzian0} \Lambda +\frac{\omega^L(u)^2}{g^L_0(u,u)}>0,
\end{equation}
(see for example \cite{CJSwind}, around formula (28)).
\end{rem}
In what follows, we will be interested  in the case that  $K$ in Proposition \ref{p_2.1} is a Killing vector field. 
%\footnote{However,  as the previous setting suggests, some of the  results in this paper could be extended to the time-dependent case. \br No se si es conveniente poner esto aqui. Parece una invitacion a que alguien lo generalice... \er}.   
In this case, all the metric elements in the proof  of 
Proposition~\ref{p_2.1} (plus the bound \eqref{lorentzian0}) are independent of the flow of $K$ and, thus, of the coordinate $t$, yielding directly:

\begin{cor}\label{cor_SSTKsplitting} For any spacetime $(L,g)$ endowed with a temporal function $t$ and a complete Killing vector field $K$ such that $\df t(K)\equiv 1$, the splitting $L=\R\times M$ in Proposition \ref{p_2.1}
can be sharpened metrically into
\begin{equation}\label{e_SSTKsplitting}
%\Phi^*(g)
g=- (\Lambda \circ \pi) \df t^2+\pi^*\omega\otimes \df t+\df
t\otimes \pi^*\omega+\pi^*g_0 ,
\end{equation}
where $\Lambda$, $\omega$ and $g_0$ are, respectively, a smooth
real function (the ``lapse''), a one form (the ``shift'') and a Riemannian metric  on $M$,
$\pi:\R\times M\rightarrow M$ is the natural projection, and
$\pi^*$ the pullback operator. Moreover, the relation
\begin{equation}
\label{lorentzian} \Lambda +|\omega|^2_{0}>0
\end{equation}
holds,  being $|\omega|_0$ the pointwise  $g_0$-norm of $\omega$.
%In this case,   the projection  $t:\R\times M\rightarrow \R$ satisfies that  $-\nabla t$ is a timelike vector field, which can be assumed future-pointing (i.e. time-orientating the spacetime) with no loss of generality.
\end{cor}

Following  \cite{CJSwind}, let us introduce the notion of SSTK spacetime. 
\begin{defi} A  spacetime  
$(L,g)$  is
{\em standard  with  a  space-transverse  Killing
vector field (SSTK)}
if it admits a (necessarily non-vanishing) complete Killing vector
field
$K$
and a spacelike hypersurface
$S$
(differentiably) transverse to
$K$
which is crossed exactly once
by every integral curve of
$K$.  
\end{defi}

\begin{cor}
A spacetime $(L,g)$ is SSTK if and only if it admits a temporal function $t$ and a complete Killing vector field $K$ such that $\df t(K)\equiv 1$. In this case, the global splitting provided by Corollary \ref{cor_SSTKsplitting} will be called 
an {\em SSTK splitting}. % and the notation in \eqref{e_SSTKsplitting} will be simplified writting $g$ instead of $\Phi^*(g)$.
\end{cor}

\begin{proof} By \cite[Proposition 3.3]{CJSwind} $(L,g)$ is an SSTK spacetime if and only if it admits an SSTK splitting and, therefore, it admits $t$ and  $K=\partial_t$ as in the statement. The converse follows from Corollary \ref{cor_SSTKsplitting}.
\end{proof}
Recall that a vector field $X$ (and, then, the full spacetime) is called {\em stationary} (resp., {\em static}; {\em  stationary-complete}; {\em   static-complete})  when it is Killing and  timelike (resp., additionally: the orthogonal distribution $X^\perp$ is involutive; $X$ is complete; both conditions 
occur)\footnote{Sometimes, our stationary spacetimes are called {\em strictly stationary} in the literature about Mathematical Relativity, and
 the name {\em stationary} is used for a Killing vector field $K$ that is timelike at some point (see for example \cite[Definition 12.2]{Lu}). %Obviously, these more general stationary spacetimes also lie under our study.  
 Indeed, the name {\em SSTK spacetime} is introduced just to avoid confusions with the previous ones: no restriction on the causal character of $K$ is assumed (whenever its flow is temporal) but the global structure must split.}. 
 When $\Lambda>0$ in \eqref{e_SSTKsplitting} the spacetime is called {\em standard stationary} and if, additionally, $\omega=0$, standard static. Any stationary or static spacetime can be written locally as a standard one. A stationary-complete   
spacetime is (globally) standard stationary if and only if it satisfies the mild causality condition of being  distinguishing  \cite{JSstat} (or, as pointed out in \cite[Prop. 2.13]{Har}, and taking into account \cite[Prop. 1.2]{Har}, future distinguishing).
%\footnote{Solo he encontrado esa proposicion con la salvedad de que exige que se satisfaga la observer-manifold condition (el cociente de las orbitas es Haussdorf). Esta se satisface siempre que el espacio sea cronologico, pero en nuestro resultado no lo asumimos y future-distinguishing no implica cronologico, verdad?} just future-distinguishing) 
 In this case, the spacetime  is not only stably causal but also causally continuous; however, the conditions to ensure that  a static-complete spacetime is standard static are more involved, see \cite{SS, Gu}.

\subsection{Conformal geometry and the appearance of Finslerian structures} \label{s_2.2}
Now, let us consider the conformal structure for an SSTK splitting \eqref{e_SSTKsplitting}. This is equivalent to compute the (future-directed) lightlike directions  and, because of $t$-independence, we can consider just the points on the slice $M=t^{-1}(0)$. Thus, the relevant vectors at each $p\in M$ can be written  with natural identifications  as $u_p=\partial_t|_p + v_p$ where $v_p\in T_pM$ and one must assume:
\begin{equation}
0=g(u_p,u_p)=-\Lambda(p)+2\omega(v_p)+g_0(v_p,v_p)=0\label{convex}
\end{equation}

\begin{figure}
\tdplotsetmaincoords{70}{110}
\begin{tikzpicture}[scale=1.8,tdplot_main_coords]
    %\draw[thick] (0,0,0) -- (4,0,0); %node[anchor=north east]{$x$};
    %\def\x{.5}
    %5\draw[thin] (0,0,0) -- ({1.2*\x},{sqrt(3)*1.2*\x},0) node[below] {$y=\sqrt{3}x$};
    \filldraw[
        draw=brown,%
        fill=brown!10,%
    ]          (0,0,0)
            -- (4,0,0)
            -- (4,4,0)
            -- (0,4,0)
            -- cycle;
       \filldraw[
        draw=blue,%
        fill=blue!10,%
    ]          (1,2,0)
            -- (2,2,0)
            -- (2,3,0)
            -- (1,3,0)
            -- cycle;
          \filldraw[
        draw=blue,%
        fill=blue!10,%
    ]          (0.5,0.5,0)
            -- (1.5,0.5,0)
            -- (1.5,1.5,0)
            -- (0.5,1.5,0)
            -- cycle;
            \filldraw[
        draw=blue,%
        fill=blue!10,%
    ]          (2.5,0.25,0)
            -- (2.5,1.25,0)
            -- (3.5,1.25,0)
            -- (3.5,0.25,0)
            -- cycle;
            %first cone
           \draw[thick] (1,1,0) circle (0.3cm and 0.15cm);
           \filldraw[blue] (1,1.1,0) circle (0.5pt);
           \draw[thick] (1,1,1) circle (0.3cm and 0.15cm);
           %\filldraw[blue] (1,1.1,1) circle (0.5pt);
           \draw[thick,->] (1,1.1,0) -- (1,1.1,1) node[left,anchor=north east]{$\partial_t$};
           \draw[thick,blue] (1,1.1,0)  -- (1+0.01,1.1+0.22,1);
           \draw[thick,magenta,->] (1,1.1,0)  -- (1-0.05,1.1-0.445,1-0.05) node[left,anchor=north east]{$u_p$};
           \draw[thick,->] (1,1.1,0)  -- (1-0.05,1.1-0.445,0) node[left,anchor=north east]{$v_p$};
           %second cone
           \draw[thick] (3,0.7,0) circle (0.3cm and 0.15cm);
           \filldraw[blue] (3,0.395,0) circle (0.5pt);
           \draw[thick] (3,0.7,1) circle (0.3cm and 0.15cm);
           %\filldraw[blue] (3,0.395,1) circle (0.5pt);
           \draw[thick,->] (3,0.395,0) -- (3,0.395,1) node[left,anchor=north east]{$\partial_t$};
           \draw[thick,magenta,->] (3,0.395,0)  -- (3,0.395+0.62,1) node[left,anchor=north west]{$u_p$};
           \draw[thick,->] (3,0.395,0)  -- (3,0.395+0.62,0) node[left,anchor=north west]{$v_p$};
           %\draw[thick,blue] (3,0.395,0)  -- (3,0.395,1);
           %third cone
           \draw[thick] (1.5,2.5,0) circle (0.3cm and 0.15cm);
           \filldraw[blue] (1.5,2.1,0) circle (0.5pt);
           \draw[thick,->] (1.5,2.1,0) --(1.5,2.1,1) node[left,anchor=north east]{$\partial_t$};
           \draw[thick] (1.5,2.5,1) circle (0.3cm and 0.15cm);
          % \filldraw[blue] (1.5,2.1,1) circle (0.5pt);
           \draw[thick,magenta,->] (1.5,2.1,0) -- ((1.5,2.1+0.71,1) node[left,anchor=north west]{$u_p$};
           \draw[thick,->] (1.5,2.1,0) -- ((1.5,2.1+0.71,0) node[left,anchor=north west]{$v_p$};
           \draw[thick,blue] (1.5,2.1,0)  -- (1.5,2.1+0.088,1);
           % orbitas de K
            \draw[dashed] (1,1.1,0) -- (1,1.1,2);
             \draw[dashed] (3,0.395,0)  -- (3,0.395,2) ;
              \draw[dashed] (1.5,2.1,0) -- (1.5,2.1,2);
              %\draw[thick,red] (1.5,2.1,0) -- (1.5+0.5,2.1+0.4,0);
           %\draw[thick,red] (1.5,2.1,0) -- (1.5-0.5,2.1+0.2,0);
   % \draw[thick] (0,0,0) -- (0,4,0); node[anchor=north west]{$y$};
    \draw[very thick,->,brown] (0,0,0) -- (0,0,2) node[left,anchor=south]{$t$};
     \draw[
        draw=red,dashed%
    ]          (0,0,1)
            -- (4,0,1)
            -- (4,4,1)
            -- (0,4,1) node[left,anchor=south]{$t=1$} 
            -- cycle;
\end{tikzpicture}
\caption{\label{figurecones} We show the lightlike vectors $u_p=\partial_t|_p+v_p$ of $(T_pL,g_p)$ with $v_p$ tangent to $M=t^{-1}(0)$. There are  three different possibilities according to the causal character of $\partial_t$. All the other lightlike vectors are proportional to one of these. The subset of vectors $v_p$ forms an ellipsoid in $T_pM$.}
\end{figure}
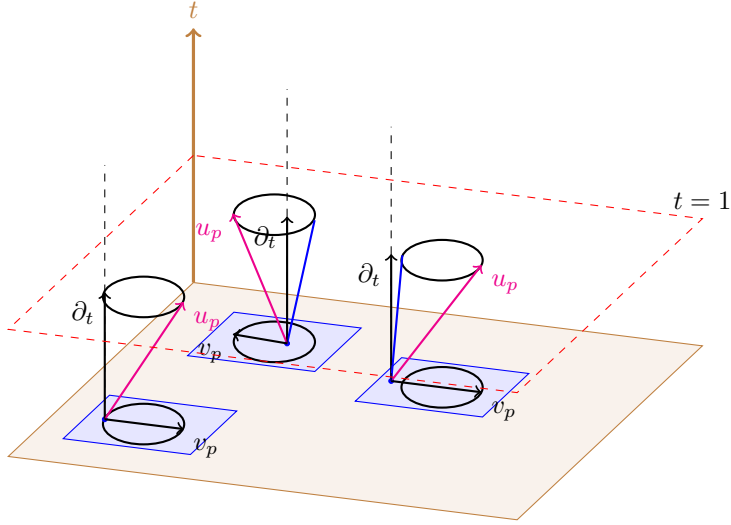

\begin{lemma} \label{l} The set $\Sigma_p$ which contains all $v_p\in T_pM$ satisfying \eqref{convex} is a $g_0$-sphere of center $W_p$, where $g_0(W_p,\cdot )=-\omega_p$, and (positive) radius $r_p:=\sqrt{\Lambda(p)+| \omega_p |^2_0}$ $= \sqrt{\Lambda(p)+| W_p |^2_0}$ (recall \eqref{lorentzian}).
\end{lemma} 
\begin{proof}
Putting $w_p=v_p-W_p$ one has:
$$
g_0(w_p,w_p)=g_0(W_p,W_p)-2 g_0(W_p,v_p)+g_0(v_p,v_p)=| \omega_p |^2_0+
2\omega(v_p)+g_0(v_p,v_p),
$$
so, \eqref{convex} holds if and only if $g_0(w_p,w_p)=\Lambda(p)+| \omega_p |^2_0$.
\end{proof}

\begin{rem}\label{r_norm} As we will be interested only in the conformal structure of the  spacetime, we can choose  the conformal metric $\tilde g= \Omega g$ were 
 $\Omega= 1/\left(\Lambda + | \omega |^2_0 \right)$, and its corresponding lapse $\tilde \Lambda$, shift $\tilde \omega$  and Riemannian metric $\tilde g_0$  will satisfy
 \begin{equation}
 \label{e_norm}  \tilde r_p   :=\sqrt{\tilde \Lambda(p)+| \tilde \omega_p |^2_{\tilde g_0}}\equiv 1 . 
 \end{equation} 
 From the definition of $W_p$, its independence of  conformal changes becomes apparent. Consistently, we can attach the conformally invariant Riemannian metric
 \begin{equation}\label{e_gr}
 g_R=g_0/\left(\Lambda + | \omega |^2_0 \right)
 \end{equation}
 to the SSTK splitting.  Indeed, then  $\tilde\Lambda=\Omega\Lambda$ and $| \tilde \omega |^2_{\tilde g_0}= \Omega | \omega |^2_0$,  thus, $\Sigma_p$ will always be given by a $g_R$-sphere of radius 1.
 \end{rem}
Let us call the pair  $(g_R,W)$ composed by a Riemannian metric $g_R$ and a vector field $W\in\frak{X}(M)$ on $M$, {\em Zermelo data}. We can summarize and systematize the previous results as follows:
\begin{prop}\label{p2.8}
(1) For each SSTK splitting \eqref{e_SSTKsplitting} there exist a smooth  hypersurface $\Sigma\subset TM
$ and Zermelo data $(g_R,W)$, both univocally determined and 
invariant under pointwise conformal 
transformations $g\mapsto \Omega g, \Omega>0$, such that:
\begin{itemize}
\item[{\it (i)}] $\Sigma$ is transverse to each tangent space $T_pM, p\in M,$ and all the lightlike directions on $M$ are spanned by the vectors 
$\partial_t + v$ such that $ v\in \Sigma$.

\item[{\it (ii)}] At each point $p\in M$, $\Sigma_p:= \Sigma \cap T_pM$ is the $g_R$-sphere of center $W_p$ and radius 1. 
\end{itemize}

(2) Conversely, for each Zermelo data  $(g_R, W)$  on $M$ there exists an SSTK-splitting  (unique up to pointwise conformal transformations) whose associated  Zermelo data by the previous point (1) are  $(g_R, W)$. 

(3) Moreover, a smooth hypersurface $\Sigma\subset TM$ can be written as the set of all the  unit $g_R$-spheres   with  center  $W_p$ at each point $p\in M$  for some Zermelo data $(g_R,W)$ if and only if $\Sigma$ is transverse to all $T_pM, p\in M$ and each $\Sigma_p:=\Sigma\cap T_pM$ is an ellipsoid  in the coordinates  induced by any basis of $T_pM$.  
\end{prop}

\begin{proof} First of all,  the transversality of a hypersurface  $\Sigma$ constructed from Zermelo data $(g_R,W)$ as in (3) can be proved as follows.   Let $F_R$ be the $g_R$-norm, namely, $F_R(v)=\sqrt{g_R(v,v)}$ for $v\in TM$.  %\sout{(recall that $F:TM\rightarrow \R$ can be seen as a Finsler metric, and it is smooth on $TM\setminus\mathbf{0}$).} 
Its indicatrix $\Sigma_R=F_R^{-1}(1)$ (i.e., the set of all its unit vectors) must be tranverse to all the tangent spaces. Otherwise, as $1$ is a regular value of $F_R$, there would be some $v\in TM \cap \Sigma_R$ such that $(\df F_R)_v(w_v)=0$ for some $w_v\in T_v (T_pM)\subset T(TM)$, where $v\in T_pM$ and $w_v$ is not 
tangent to the unit $g_R$-sphere on $T_pM$. Thus, all the vectors tangent to $T_pM$, when looked as elements of (the vertical space in) $T_v(TM)$ lie in the kernel of $\df (F_R)_v$. In particular, this happens to $v_v:= \df(v+\lambda v)/\df t|_0$, so, 
$$
0= (\df F_R)_v(v_v)= \left.\frac{\df F_R(v+\lambda v)}{\df\lambda}\right|_{\lambda=0}=\left. \frac{\df (1+\lambda)}{\df\lambda}\right|_{\lambda=0} F_R(v)=1,
$$  
a contradiction. Then,  notice that  the  pointwise translation $TM\rightarrow TM$, $u_p\mapsto u_p+W_p$  provided by the vector field $W$ does preserve the smooth fiber bundle struture of $TM$ (namely, its structure as an affine bundle, even though not as a linear bundle).  Therefore,  $\Sigma=\Sigma_R+ W$ must remain transverse, as required.

Now,  part (1) follows just  by applying pointwise the previous lemma and remark, and (2) by constructing the SSTK splitting with $g_0=g_R, \omega=-g(W,\cdot )$ and $\Lambda= 1-g_R(W,W)$. For (3) the necessary condition is now straightforward, and the sufficient one follows by taking  $W$ as the centroid of each ellipsoid and the hypersurface $\Sigma-W$ as the unit sphere bundle for\footnote{The smoothness of the Zermelo data $(g_R,W)$ follows from  the transversality of $\Sigma$, \cite[Proposition 2.15]{CJSwind}; see also  Section 2.2 of this reference (especially around Example 2.16) for further discussions on transversality.}  $g_R$. 
\end{proof}

\begin{defi}
A {\em wind Riemannian structure (WRS)} on a manifold is any hypersurface $\Sigma$ embedded in $TM$ which satisfies the equivalent conditions in Proposition~\ref{p2.8}~(3), namely, for some (univocally determined) Zermelo data $(g_R,W)$, one has $\Sigma= S_R+W$, where $S_R$ is the indicatrix (unit sphere bundle) of $g_R$.
\end{defi}

 Such a definition  admits natural extensions: $\Sigma\subset M$ is a {\em wind Finslerian structure} when $\Sigma= S+W$, where $S$ is the indicatrix for a Finsler metric $F$ (now, $(F,W)$ is determined univocally if one imposes additionally that $W$ provides the  centroid  of $S$ at each point);  
then $\Sigma_p$ is a {\em  wind Minkowski structure} at each $p\in M$. The properties of wind Finslerian structures and norms as well as the relation with classical {\em Zermelo navigation problem} are studied extensively in \cite{CJSwind} (see also \cite{JS}).

 Let us  recall the following particular cases  for  an SSTK splitting. 
%\begin{enumerate}
%\item 
\subsubsection{\bf Static case} $\omega\equiv 0$. Now, $W\equiv 0$,  $g_R=g_0/\Lambda$ (necessarily $\Lambda>0$) and the spacetime is conformal to the product $(\R\times M, -dt^2+g_R)$. It is well-known that the global  hyperbolicity of the spacetime is equivalent to the completeness of $g_R$ as well as to the fact that $M$ is a Cauchy  hypersurface, \cite[Th. 3.67, 3.69]{BEE}. Moreover, the spacetime is always causally continuous, and it will be causally simple if and only if $g_R$ is convex (see the next case).

%\item 
\subsubsection{\bf Stationary case} \label{s222} $\Lambda>0$ ($K$ timelike). Now, $\Sigma_p$ is a $g_0$-sphere of center $W_p$ and radius $r_p> \| W_p\|_0$ (recall Lemma \ref{l}), that is, the $g_R$-norm of $W_p$  is  smaller than one. Therefore, the zero tangent vector is always included in the interior of each sphere $\Sigma_p$ and the hypersurface $\Sigma$ can be regarded as the indicatrix of a Finsler metric $F$ of Randers type; concretely,
\begin{equation} \label{e_Randers}
F(v)  =
\frac{\omega(v)}{\Lambda}+\sqrt{
\frac{g_0(v,v)}{\Lambda}+\left( 
\frac{\omega(v)}{\Lambda}\right)^2}.
\end{equation}
Then, the (future-directed) SSTK lightlike directions on $M$ are neatly described by the vectors $\partial_t+v$ such that $F(v)=1$.

Recall that, in general, a Finsler metric is {\em not reversible}, that is, $F$ behaves as a pointwise norm which is only positive
homogeneous ($F(\lambda v)=|\lambda| F(v)$ is ensured only for $\lambda\geq 0$). As a consequence, $F$ induces a  (possibly non-symmetric)  {\em
generalized distance} $d_F$ and one must distinguish between {\em forward} open balls $B^+_F(p,r)=\{q\in M: d_F(p,q) < 
r\}$ and {\em backward} ones $B^-_F(p,r)=\{q\in M: d_F(q,p) < r\}$. Moreover, even 
though geodesics   make the usual natural sense, the reverse parametrization of a geodesic may not be a geodesic. 

Standard stationary spacetimes were systematically studied in \cite{CJS} by using \eqref{e_Randers} (choosing a conformal representative so that $\Lambda\equiv 1$),  including its causal hierarchy. 
Indeed,  the existence of the temporal function $t$ implies that  standard stationary  spacetimes are stably causal. The other higher  steps in the standard ladder of causality %for standard stationary spacetimes 
are neatly characterized by the Randers metric $F$ as follows   (see \cite{CJS}): 
\begin{enumerate}
\item  $M=t^{-1}(0)$ is a Cauchy hypersurface if and only if $F$ is complete.
\item $\R\times M$ is globally hyperbolic if and only if the intersections of the closed balls $\bar B^+_F(p,r)\cap \bar B^-_F(p',r')$ are compact for all $p,p'\in M, r,r'>0$.  
\item $\R\times M$ is causally simple if and only if $F$ is convex,  in the sense that for each $(p,q)\in M\times M$ there exists a (non-necessarily unique) minimizing $F$-geodesic from $p$ to $q$.
\item    $\R\times M$ is always  causally continuous, \cite{JSstat}.
\end{enumerate}
It is also worth pointing out that the causal boundary of such a stationary spacetime can be described in terms of the natural (Cauchy, Gromov) boundaries of the Finslerian manifold $(M,F)$ (see \cite{FHS} for a thorough study).  On the other hand, a description of the conformal maps of the stationary spacetime in Finslerian terms can be found in \cite{JLP15} and some links between the flag curvature of the Randers metric and the conformal invariants of the spacetime are developed in \cite{GHWW}. 
%\item 
\subsubsection{\bf Nonnegative lapse\label{s223}}  $\Lambda\geq 0$ ($K$ causal). When $\Lambda(p)=0$ then the $g_R$-norm of $W_p$ is equal to 1 and $\Sigma_p$ contains the zero vector. Then, $F$ becomes a {\em Kropina} metric; this is a singular type of Finslerian metric $F(v)=-g(v,v)/(2\omega(v))$ which applies only to $v\in T_pM$ such that $\omega(v)<0$. Indeed, one can rewrite \eqref{e_Randers} as:
\begin{equation}\label{randers-kropina}
 F(v) = \frac{g_0(v,v)}{-\omega(v)+\sqrt{\Lambda
 g_0(v,v)+\omega(v)^2}},\quad\quad 
 %\quad \forall v\in  A,
 \end{equation}
 which makes sense even when $\Lambda$ vanishes and will be called a {\em Randers-Kropina metric}. Now, the SSTK lightlike directions on $M$ are again described by the vectors of the form $\partial_t+v$ with $F(v)=1$  
with the caution that, whenever $\Lambda=0$, the direction $\partial_t$ must be included (as $0\in \Sigma_p$)  and $F$ is applied only on vectors $v$ with $\omega(v)<0$. 

The Randers-Kropina metric  defines an {\em $F$-separation} $d_F$ formally analogous to the generalized distance of the 
stationary case. Its properties   are carefully studied in \cite[Section 4]{CJSwind}. Some important differences between the $F$-
separation $d_F$ and the generalized distance in the standard stationary case are: (i) $d_F(p,q)$ is 
infinite if there is no admissible curve $\alpha$  from $p$ to $q$ (where {\em admissible} means here satisfying $\omega(\alpha'(s))<0$ whenever $\Lambda(\alpha(s))=0$), and (ii) the continuity of $d_F$ is ensured only outside the diagonal; indeed, $d_F$ is discontinuous on $(p,p)$ whenever $d_F(p,p)>0$ (and in this case $\Lambda(p)=0$ necessarily), \cite[Th. 4.5, Prop. 4.6]{CJSwind}.  However, Randers-Kropina metrics admit geodesics analogous to the Finslerian ones. 

The  ladder of causal properties of the spacetime can be characterized in  
terms of the properties of $F$-separation, and it becomes  formally analogous to the conclusions (a)---(d) of the stationary case, \cite[Th. 4.9]{CJSwind}.
%\item 
\subsubsection{\bf General SSTK case} ($K$ may be spacelike). Now, when $\Lambda(p)<0$ one has $g_R(W_p,W_p)>1$, that is, the zero vector  is not included in the solid ellipsoid enclosed by  $\Sigma_p$. The half lines starting at $0$ and tangent to the $g_R$-sphere  $\Sigma_p$, provide 
   a cone. Such a cone is tangent to $\Sigma_p$ in an $(n-1)-$sphere $S^{n-1}_p$, and $\Sigma_p\setminus S^{n-1}_p$ has two connected pieces. One of them is convex (when 
   looked  inside the cone from infinity), and can be described as the open set containing the vectors $v$  inside the cone such that $F(v)$ 
   (computed by using the expression~\eqref{randers-kropina}) is equal to $1$, see Figure \ref{figurecones2}. 
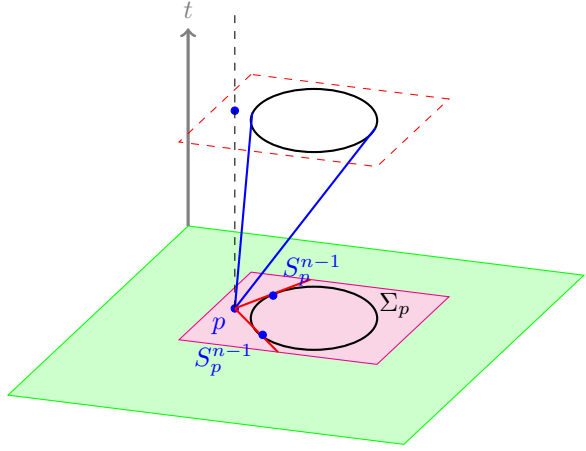
\begin{figure}
\tdplotsetmaincoords{70}{110}
\begin{tikzpicture}[scale=2.8,tdplot_main_coords]
    %\draw[thick] (0,0,0) -- (4,0,0); %node[anchor=north east]{$x$};
    %\def\x{.5}
    %5\draw[thin] (0,0,0) -- ({1.2*\x},{sqrt(3)*1.2*\x},0) node[below] {$y=\sqrt{3}x$};
    \filldraw[
        draw=green,%
        fill=green!20,%
    ]          (0.5,1.5,0)
            -- (3,1.5,0)
            -- (3,3.5,0)
            -- (0.5,3.5,0)
            -- cycle;
       \filldraw[
        draw=magenta,%
        fill=magenta!20,%
    ]          (1,2,0)
            -- (2,2,0)
            -- (2,3,0)
            -- (1,3,0)
            -- cycle;
%          \filldraw[
%        draw=blue,%
%        fill=blue!20,%
%    ]          (0.5,0.5,0)
%            -- (1.5,0.5,0)
%            -- (1.5,1.5,0)
%            -- (0.5,1.5,0)
%            -- cycle;
%            \filldraw[
%        draw=blue,%
%        fill=blue!20,%
%    ]          (2.5,0.25,0)
%            -- (2.5,1.25,0)
%            -- (3.5,1.25,0)
%            -- (3.5,0.25,0)
%            -- cycle;
            %first cone
%           \draw[thick] (1,1,0) circle (0.3cm and 0.15cm);
%           \filldraw[blue] (1,1.1,0) circle (0.5pt);
%           \draw[thick] (1,1,1) circle (0.3cm and 0.15cm);
%           \filldraw[blue] (1,1.1,1) circle (0.5pt);
%           \draw[thick,blue] (1,1.1,0)  -- (1+0.01,1.1+0.22,1);
%           \draw[thick,blue] (1,1.1,0)  -- (1-0.05,1.1-0.445,1-0.05);
           %second cone
%           \draw[thick] (3,0.7,0) circle (0.3cm and 0.15cm);
%           \filldraw[blue] (3,0.395,0) circle (0.5pt);
%           \draw[thick] (3,0.7,1) circle (0.3cm and 0.15cm);
%           \filldraw[blue] (3,0.395,1) circle (0.5pt);
%           \draw[thick,blue] (3,0.395,0)  -- (3,0.395+0.62,1);
%           \draw[thick,blue] (3,0.395,0)  -- (3,0.395,1);
           %third cone
           \draw[thick] (1.5,2.5,0) circle (0.3cm and 0.15cm);
           \node at (1.2,2.8,0) (s) {$\Sigma_p$};
           \filldraw[blue] (1.5,2.1,0) circle (0.5pt)  node[anchor=north east]{$p$};
           \draw[thick] (1.5,2.5,1) circle (0.3cm and 0.15cm);
           \filldraw[blue] (1.5,2.1,1) circle (0.5pt);
           \draw[thick,blue] (1.5,2.1,0) -- ((1.5,2.1+0.71,1);
           \draw[thick,blue] (1.5,2.1,0)  -- (1.5,2.1+0.088,1);
           % orbitas de K
           % \draw[dashed] (1,1.1,0) -- (1,1.1,2);
            % \draw[dashed] (3,0.395,0)  -- (3,0.395,2) ;
              \draw[dashed] (1.5,2.1,0) -- (1.5,2.1,1.5);
              \draw[thick,red] (1.5,2.1,0) -- (1.5+0.5,2.1+0.4,0);
              \filldraw[blue] (1.5+0.3,2.1+0.25,0) circle (0.5pt) node[anchor=north east]{$S^{n-1}_p$};
           \draw[thick,red] (1.5,2.1,0) -- (1.5-0.5,2.1+0.2,0);
           \filldraw[blue] (1.5-0.23,2.1+0.11,0) circle (0.5pt) node[anchor=south west]{$S^{n-1}_p$};
   % \draw[thick] (0,0,0) -- (0,4,0); node[anchor=north west]{$y$};
    \draw[very thick,->,gray] (0.5,1.5,0) -- (0.5,1.5,1) node[left,anchor=south]{$t$};
     \draw[
        draw=red,dashed%
    ]           (1,2,1)
            -- (2,2,1)
            -- (2,3,1)
            -- (1,3,1)
            -- cycle;
            -- cycle;
\end{tikzpicture}
\caption{\label{figurecones2} The diagram shows the $g_R$-sphere $\Sigma_p$ determined by the projections of lightlike vectors with the time coordinate equal to $1$. When $\Lambda(p)<0$, the half lines starting at $0$ and intersecting $\Sigma_p$ form a conic region. The intersection of the boundary of this conic region with $\Sigma_p$ is an $(n -1)$-dimensional $g_R$-sphere $S^{n-1}_p$  (in the figure, only two points) 
which divides $\Sigma_p$ in two connected components.}
\end{figure}   
   The other connected part is computed analogously by putting $F_l(v)=1$ where, now,
\begin{equation}\label{e_lorentz_Finsler}
 F_l(v) = -\frac{g_0(v,v)}{\omega(v)+\sqrt{\Lambda
 g_0(v,v)+\omega(v)^2}} \; (=-F(-v)).
 %\quad\quad 
 %\quad \forall v\in  A.
 \end{equation}
The part $F(v)\equiv 1$ in the regions $\Lambda<0$ and $\Lambda\geq 0$ matches naturally. So, $F$ behaves as a {\em conic} Finsler metric on all $M$; indeed the  conic region where $F$ is defined is properly the interior of a cone when $\Lambda(p)<0$, an open half plane when $\Lambda(p)=0$ and all $T_pM$ otherwise (see \cite{JSpisa} for a systematic study of conic Finsler metrics). However, the concaveness of $F_l(v)\equiv 1$ makes $F_l$ to behave as a Lorentz-Finsler metric, in the sense that it  yields a reverse triangle inequality similar to the Lorentzian one, that is, $F_l(v+w)\geq F_l(v)+F_l(w)$ for all $v,w\in T_pM$ in the conic domain of definition. 

So, the SSTK lightlike directions on  the hypersurface  $M$ are described by the vectors of the form $\partial_t+v$ taking into account that, when $\Lambda(p)<0$, one has to choose vectors with either $F(v)=1$ or $F_l(v)=1$, including those in $S_p^{n-1}$ (which corresponds with the limit case $F(v)=F_l(v)=1$).
Again, the causal properties of the SSTK spacetime can be described by using the Finslerian elements $F, F_l$ or, directly, by means of the hypersurface $\Sigma$. However, 
this general  case  is much subtler than the previous ones, and   it will be   sketched in Section \ref{s_3}.    
%\end{enumerate} 

\subsection{An application: Finslerian consequences of Harris' stationary quotients} \label{s_2.3}
Even though we will focus on properties of general SSTK spacetimes linked to wind Riemannian structures,  we emphasize now the links between the  conformal geometry of standard stationary spacetimes and the geometry of Randers spaces, with applications also to arbitrary Finsler manifolds. Apart from the applications explained in the point (2) of subsection \ref{s_2.2}, several links  introduced in \cite{CJS}  include the behavior of completeness under projective changes (see below) and properties on the differentiability of the distance function to a closed subset (see also \cite{TaSa}) as well as on the Hausdorff measure of the set of cut points. 
Now, we can add a new application by translating  a recent result by Harris on group actions \cite{Har} on static/stationary spacetimes to the Finslerian setting, namely: %the translation to Randers manifolds is the following:
\begin{quote}
{\em There exists a Randers manifold $(M,R)$ such that  not all  its closed symmetrized balls are  compact but its universal covering $(\tilde M,\tilde R)$ satisfies that  all  its closed symmetrized balls are compact.}
\end{quote}
In order to understand the subtleties of this result, recall the following. 
\begin{rem}
(1) Of course, such a property cannot hold in the Riemannian case $(M,g_R)$, because the closed $g_R$-balls are compact if and only $g_R$ is complete, and this property holds if and only if the universal covering $(\tilde M,\tilde g_R)$ is complete.

(2) The key in the Randers case is that the compactness of the closed symmetrized $R$-balls (which is equivalent to the compactness of the intersections between closed forward and backward $R$-balls) does not imply $R$-geodesic completeness (nor the completeness of $R$ in any of the equivalent senses of the Finslerian Hopf-Rinow result).  
Indeed, as shown in \cite{CJS} for Randers metrics (and then extended to the general Finslerian case in \cite{Matveev}), the compactness of the closed symmetrized $R$-balls is equivalent to the existence of a complete Randers metric $R^f$ which is related to $R$ by means of a {\em trivial projective transformation} (i.e., $R^f=R+df$ for some function $f$ on $M$ such that $R+df>0$ on $TM\setminus\mathbf{0}$).  
As will be apparent  below,  the existence of such a function  $\tilde f$  in the universal covering $\tilde M$ does not imply  the existence of  an analogous function $f$  in the manifold $M$,  as $\tilde f$   is not necessarily projectable. 
\end{rem}
Next, let us derive the Randers result from Harris' ones.
\begin{exe} Let start with 
\cite[Example 3.4(b)]{Har}, which exhibits a globally hyperbolic 
standard static spacetime $L'$ admitting a group of isometries 
$G$ such that the quotient $L=L'/G$ is static-complete (i.e. 
it admits a complete static vector field) and causally 
continuous, but not globally hyperbolic.  A static-complete 
causally continuous spacetime is not necessarily a standard 
static spacetime; nevertheless, as any distinguishing stationary-complete space is standard stationary \cite{JSstat},  so is $L
$ too. Then, we can write the standard stationary splitting $L=\R
\times M$ with associated Randers metric $F$, as explained in subsection \ref{s_2.2}. As $L$ is not 
globally hyperbolic, $(M,F)$ cannot satisfy the property of 
compactness of closed symmetrized balls. However, its 
universal Lorentzian covering $\tilde L=\R\times \tilde M$ 
inherits a Randers metric $\tilde F$ which must satisfy such a property of compactness (indeed,  $\tilde L$ must also be the universal covering of $L'$ and, so, globally hyperbolic). We emphasize that, being $\tilde L$ static-complete and simply connected, it can be written as a standard static spacetime,  \cite[Theorem 2.1(1)]{San};  however,  the splitting $\tilde L=\R\times \tilde M$ we are  using is  only  standard stationary but not standard static (otherwise, $\tilde F$ would be Riemannian).
\end {exe}

\section{A new characterization of WRS completeness}\label{s_3}

\subsection{WRS geodesics and completeness vs SSTK Cauchy hypersurfaces} \label{s_3.1}
In what follows, let $\Sigma$ be a WRS on $M$ with associated Zermelo data $(g_R,W)$ and SSTK splitting $(\R\times M, g)$, being $K=\partial_t$  Killing, $\omega =-g(K, \cdot )$ and the conformal normalization $
\Lambda + | \omega |^2_0\equiv 1$ chosen, so that $g_0=g_R$, according  to  \eqref{e_norm} and \eqref{e_gr}. 
%Remark \ref{r_norm}, 
 %\eqref{e_norm}. 

\subsubsection{First definitions} The region where 
$K$ is, resp., timelike ($\Lambda>0$), lightlike ($\Lambda=0$) or spacelike ($\Lambda<0$) will be called of   
{\em mild wind} (as  $\lVert W \rVert _R<1$), {\em critical wind} and 
{\em strong wind}; the latter will be denoted $M_l$ as it 
contains  the  Lorentz-Finsler metric  $F_l$ (besides  the conic Finsler one $F$, defined on all $M$). Moreover, we define the (possibly signature-changing) metric
\begin{equation}
\label{e_h}
h= \Lambda g_0 + \omega\otimes \omega,
\end{equation}
which is  Riemannian (resp. degenerate, of signature $
(+,          -,\dots , -)$) in the mild (resp. critical, 
strong) wind region. Remarkably, the conformal metric 
$h/\Lambda$ in the region $\Lambda\neq 0$ is the induced 
metric on the orbit space of $\R\times M$ obtained by 
taking the quotient by the  flow of $K$ (see 
\cite[Prop. 3.18]{CJS}); clearly, $h/\Lambda$ is not 
conformally invariant (however, $h/\Lambda^2$ is). In 
any case $-h$ (as well as $h/\Lambda$ and $-h/\Lambda^2$) 
becomes a Lorentzian metric on $M_l$, and it is 
naturally time-oriented because the future timelike 
vectors of $\R\times M$ project onto a single timecone 
of $-h$. 

Let 
us define the following subsets of $TM$. First,  put 
$$
A\cup A_E:=\{v\in TM: v \; \hbox{is the projection of a lightlike vector of the SSTK splitting}\}.
$$
 Inside this set, we will consider two non-disjoint subsets $A$ and $A_E$. The set $A$ will contain the interior of all the conic domains where the conic Finsler metric $F$ in \eqref{randers-kropina} is naturally defined; then, $A_l:=TM_l \cap A$ will  also be the natural open domain for the Lorentz-Finsler metric $F_l$ in \eqref{e_lorentz_Finsler}.
The set  $A_E$ will contain  all the continuously extended domain of  $F_l$ in $A_l$, plus the zeroes of the critical region; then, $A\cup A_E$ can  also be regarded as the extended domain for $F$ plus the zeroes of the critical region.  Explicitly, $A_p= A\cap T_pM$ and  $(A_E)_p= A_E \cap T_pM$ are, at each $p\in M$: 
$$\begin{array}{rl}
v_p\in A_p \Longleftrightarrow &
\left\{ 
\begin{array}{ll}
v_p\neq 0 & \hbox{when} \; \Lambda(p)>0 \\
-\omega_p(v_p) (=g_R(W_p,v_p)) > 0 & \hbox{when}\; \Lambda(p)=0\\
-\omega_p(v_p)>0 \;\; \hbox{and} 
\;\; h(v_p,v_p)>0 & \hbox{when} \; \Lambda(p)<0
\end{array}
\right. \\
v_p\in (A_E)_p \Longleftrightarrow & 
\left\{ 
\begin{array}{ll}
%  \bb % -\omega_p(v_p)\geq 0 \eb
  %\;\; \hbox{or} \;\; \er v_p = 
  0_p 
  &  \hbox{when}\; \Lambda(p)=0\\
-\omega_p(v_p)>0 \;\;  \hbox{and} 
\; \; h(v_p,v_p)\geq 0 & \hbox{when} \; \Lambda(p)<0
\end{array}
\right. 
\end{array}
$$
%\footnote{\br Miguel: creo que en $A_E$ solo esta el cero de la region critica\er}
\begin{rem}\label{r_dominios}
 $F$ and $F_l$ can be extended continuously from their  open domains $A$, $A_l$ 
 to $A\cup (A_E\cap TM_l)$ and  $A_E\cap TM_l$, respectively.
Indeed, the metric  $h$ (see \eqref{e_h}) 
vanishes in  $(A_E\cap TM_l)\setminus A_l$,  making equal 
the expressions  \eqref{randers-kropina}, 
\eqref{e_lorentz_Finsler} for $F$ and $F_l$ there.  However, the introduced notation will take also  into account the following subtlety which occurs for the zeroes of the critical region. 

We will work 
typically with lightlike curves
 $\tilde \alpha(t)=(t,\alpha(t))$ in the 
spacetime parameterized with the $t$-coordinate and then, necessarily,  either 
$F(\alpha'(t))\equiv 1$  or $F_l(\alpha'(t))\equiv 1$   whenever $\alpha'$ does not vanish. 
However, when $\tilde \alpha$ is parallel to $K$ at some point $
\tilde \alpha(t_0)=(t_0,p)$, $p\in M$, then $p$ belongs to the critical 
region and $\tilde \alpha'$ projects onto the zero vector $
\alpha'(t_0)=0_p$, which belongs to the indicatrix $\Sigma_p$. 
Clearly, $F_p$  and $(F_l)_p$  cannot be extended continuously to $0_p$; 
however, whenever   the $F$-length (resp. $F_l$-length) of lightlike curves as $\tilde \alpha$ above is 
considered, we will put $F(0_p)=1$ (resp. $F_l(0_p)=1$) as this is the continuous 
extension of the function  $t\rightarrow F(\alpha'(t))$ (resp. $t\rightarrow F_l(\alpha'(t))$). 

 Summing up, we adopt the following convention: 
the {\em extended domain} of $F$ and $F_l$  is all $A\cup A_E
$, with $F(0_p)=F_l(0_p)=1$ whenever $\Lambda(p)=0$ (even if $F$ and $F_l$ are   not continuous there)  and with $F_l(v_p)=\infty$ on the  mild and critical regions when  $v_p\not=0$. 
%\sout{of mild wind  $\Lambda\geq 0$,  that is, when $v_p\neq 0$ with $\Lambda(p)>0$. } \er
\end{rem}
\begin{rem}\label{Riembounded}  Both pseudo-Finsler metrics $F$ and $F_l$ are Riemannianly lower bounded (according to \cite[Definition 3.10]{JSpisa}), that is, there exists a Riemannian metric $h_R$ in $M$ such that $\sqrt{h_R(v,v)}\leq F(v),F_l(v)$ for every $v\in A_E\cup A$. 
This is observed in \cite[below Definition 2.24]{CJSwind} and it is  easy to check directly, because it is sufficient to prove that the property holds locally (see \cite[Remark 3.11 (1)]{JSpisa}). 
\end{rem}
\subsubsection{Wind curves and balls}
A (piecewise smooth) curve $\tilde \alpha(t)=(t,\alpha(t))$ is causal (necessarily future-directed) in the SSTK splitting if and only if the velocity $\alpha'(t)$ is  {\em Zermelo-bounded},  in the sense that $\alpha'(t)$ always 
belongs to the closed $g_R$-ball of $ T_{\alpha(t)}M$ with center $ W_{\alpha(t)}$ and 
radius 1. Even though the curve $\tilde \alpha$ 
is always regular, the velocity of $\alpha$ may vanish either 
in the region of mild wind or in the region of critical wind. 
Nevertheless,  in Riemannian Geometry,  it is natural to work with regular curves in order to  reparametrize  all the curves with the arc-parameter. 
 In our approach, one can avoid the use of non-regular curves $\alpha$ in the region of mild wind but, for 
critical wind, the appearance of curves  with vanishing 
velocity becomes unavoidable; indeed, as emphasized in the 
definition of $A\cup A_E$, the zero vector $0_p$ is the 
projection of a lightlike vector if (and only if) $p$ lies in the critical region.  
 Accordingly, we will say that a (piecewise smooth) curve $\alpha$ is 
  {\em $\Sigma$-admissible} if its velocity lies in $A\cup A_E$ and  a {\em wind curve} if, additionally, it is Zermelo bounded.  
 Even though we will work  with  piecewise smooth curves as usual,  the order of differentiability can be lowered in a natural way. This happens  when considering causal continuous curves \cite[p. 54]{BEE}; in this case, 
 the natural assumption would  be  to  consider locally $H^1$-curves (for causal curves and, then, for wind ones), because, under this regularity, being future-directed continuous causal becomes  equivalent to the existence of an almost everywhere future-directed causal velocity \cite[Th. 5.7]{CaFlSa}).

The previous definitions yield directly the following characterization of wind curves. 
  
\begin{prop}\label{p_uff}
A  (piecewise smooth)  curve $\alpha: I\subset \R\rightarrow M$,  with $I$ an interval,  
is  a wind curve if and only if its {\em graph} $\tilde \alpha$  in  the associated SSTK splitting  defined as
$$
\tilde\alpha(t)=(t, \alpha(t)), \qquad \quad \forall t\in I
$$
is a  causal  curve. In this case,  
\begin{equation}
\label{e_uff}
\ell_F(\alpha|_{[a,b]})\leq b-a \leq \ell_{F_l}(\alpha|_{[a,b]}).
\end{equation} 
for each $a, b\in I$, with $a<b$.
\end{prop}  

Now, for each $p,q\in M$, let   
  $$C^{\Sigma}_{p, q}= \{\hbox{wind  curves starting at} \; p \; \hbox{and ending at} \; q\}.$$
The {\em forward} and {\em backward} {\em wind balls}  of center $p_0\in M$ and radius $r>0$ 
 associated with the WRS $\Sigma$ are, resp:
\begin{align*}
&B^+_{\Sigma}(p_0,r)=\{x\in M:\ \exists\   \gamma\in C^{\Sigma}_{p_0, x}, \text{ s.t. }   r=b_\gamma-a_\gamma \, \text{and} \;  \ell_F(\gamma)<r<\ell_{F_l}(\gamma)\},\\
&B^-_{\Sigma}(p_0,r)=\{x\in M:\ \exists\   \gamma\in C^{\Sigma}_{x, p_0}, \text{ s.t. }  r=b_\gamma-a_\gamma \, \text{and} \;  \ell_F(\gamma)<r<\ell_{F_l}(\gamma)\},\\
\intertext{where $a_\gamma$ and $b_\gamma$ are the endpoints of the interval of definition of each $\gamma$. These balls are open \cite[Remark 5.2]{CJSwind} and their closures are called {\em (forward, backward) closed wind balls}, denoted  $\bar{B}^+_{\Sigma}(p_0,r), \bar{B}^-_{\Sigma}(p_0,r)$.  Between these two types of  balls,  the {\em forward} and {\em backward  c-balls} are defined, resp., by:} 
&\hat{B}^+_{\Sigma}(p_0,r)=\{x\in M:\ \exists\  \gamma\in C^{\Sigma}_{p_0, x},\text{ s.t. } r=b_\gamma-a_\gamma \, \text{(so,} \;  \ell_F(\gamma)\leq r\leq\ell_{F_l}(\gamma) )  \}
%\cup\{p_0:\text{ if $0_{p_0}\in \Sigma_{p_0}$}\}
,\\
&\hat{B}^-_{\Sigma}(p_0,r)=\{x\in M:\ \exists\ \gamma\in C^{\Sigma}_{x, p_0},\text{ s.t. } r=b_\gamma-a_\gamma \, \text{(so,} \; \ell_F(\gamma)\leq r\leq\ell_{F_l}(\gamma))  \} 
%\cup\{p_0:\text{ if $0_{p_0}\in \Sigma_{p_0}$}\}, 
 \end{align*}
for  $r> 0$; for $r=0$, by convention $\hat{B}^\pm_{\Sigma}(p_0,0)=p_0$ (so that, 
consistently with our conventions, if $0_{p_0}\in \Sigma_{p_0}$, then $p_0 \in \hat B^\pm_\Sigma(p_0,r)$ for all $r\geq 0$).  When $\Sigma$ is the indicatrix of  a Randers metric,  then  $B^\pm_{\Sigma}(p_0,r)$
coincides with the usual (forward or backward) open balls. Such balls have a neat interpretation \cite[Prop. 5.1]{CJSwind}:
\begin{prop}\label{bolas2} For the associated SSTK  splitting:
 \begin{align*}
&I^+(0,x_0)= \cup_{s> 0}\{0+s\}\times B_{\Sigma}^+(x_0,s),\\
&I^-(0,x_0)= \cup_{s> 0}\{0-s\}\times B_{\Sigma}^-(x_0,s),\\
&J^+(0,x_0)= \cup_{s\geq  0}\{0+s\}\times \hat{B}_{\Sigma}^+(x_0,s),\\
&J^-(0,x_0)= \cup_{s\geq  0}\{0-s\}\times \hat{B}_{\Sigma}^-(x_0,s).
\end{align*}
\end{prop}
It is worth emphasizing that the c-balls make a proper natural sense even for a Riemannian metric (see \cite[Example 2.28]{CJSwind}). Indeed, the property of closedness for all the forward and backward c-balls, called {\em w-convexity}, extend naturally the notion of convexity for Riemannian and Finslerian manifolds, and becomes equivalent to the closedness of all $J^\pm(t_0,x_0)$ above and, thus, to the causal simplicity of the SSTK spacetime, \cite[Theorems 5.9, 4.9]{CJSwind}.

\subsubsection{Geodesics}\label{s_313} Starting at the notions of balls on the WRS, geodesics can be defined as follows.  A   wind  curve $\gamma: I=[a,b]\to M$,  $a<b$,   is called a
{\em unit extremizing geodesic} if
\begin{equation}\label{eunitgeodesic}
 \gamma(b)\in \hat{B}_\Sigma^+(\gamma(a),b-a)\setminus B_\Sigma^+(\gamma(a),b-a) .  \end{equation}
%for every $t\in (a,b]$. 
Then, a curve is an {\em extremizing geodesic}  if it is an affine    %increasing 
reparametrization  of a unit extremizing geodesic, and it is a {\em geodesic} if it is locally an {\em extremizing geodesic}.

However,  geodesics for a WRS can be looked simply as the  projections on $M$ of the future-directed lightlike pregeodesics of the associated SSTK spacetime $(\R\times M,g)$ parametrized proportionally to the $t$ coordinate \cite[Th. 5.5]{CJSwind}. As $K$ is Killing, any spacetime  lightlike geodesic $\rho$  has the relevant invariant $C_{\rho} =g(\rho'(t),K)$. Reparametrizing with the $t$ coordinate, we have a 
lightlike pregeodesic $ \tilde\gamma(t)  =(t,\gamma(t))$ and its projection   is a unit $\Sigma$-geodesic  which belongs  to one of the following cases: 

\begin{enumerate}
\item[(1)] $C_\rho<0$: $\gamma$ is a geodesic of the conic Finsler metric $F$  on $M$, and $\gamma'(t)$ lies always in $A$.

\item[(2)] $C_\rho>0$: $\gamma$ is a geodesic of the Lorentz-Finsler metric $F_l$ on $M_l$, and  $\gamma'(t)$ lies always in $A_l:= TM_l\cap A$.

\item[(3)] $C_\rho=0$. We have two subcases: 

(3a) $\gamma$ is an {\em exceptional geodesic},  constantly equal to some point $p_0$ 
with $\Lambda (p_0)=0$ and $d\Lambda_{p_0}$ vanishing on the kernel of $\omega_p$, and 

(3b)  $\gamma$ is a  {\em boundary geodesic},  included in the closure of $M_l$ and it satisfies: 
(i) whenever $\gamma$ remains in $M_l$,  it is 
a lightlike  pregeodesic of the Lorentzian metric $h$  in \eqref{e_h},  reparametrized so that $F(\gamma')\equiv  F_l(\gamma')$ is a constant $c=1$, and 
(ii) $\gamma$ can reach the boundary $\partial M_l$ (which is included in the critical region $\Lambda=0$) only at  isolated points $s_j\in I, j=1,2,...$, where
$\gamma'(s_j)=0$ (in this case, $d\Lambda$ does not vanish on all the Kernel of $\omega_{\gamma(s_j)}$).   
\end{enumerate}

 Even if normal neighborhoods do not make sense in general for WRS's, all the geodesics departing from a given point $x_0$ (or more generally from  points close  to $x_0$) have length uniformly bounded from below by a positive constant.
\begin{prop}\label{normalneigh}
Let $(M,\Sigma)$ be a wind Riemannian structure and $x_0\in M$, then there exists $\varepsilon >0$ and a neighborhood $U_0$ of $x_0$, such that 
 the unit $\Sigma$-geodesics 
departing from $x\in  U_0$   are defined on
$[0,\varepsilon)$  and they are  extremizing.% and therefore minima (resp. maxima) of $\ell_F$ (resp. $\ell_{F_l}$) with respect to any variation. 
\end{prop}
\begin{proof}
The same proof of \cite[Proposition 6.5]{CJSwind} works in this case just by replacing $x_0$ with $x\in U_0$  where, as in that proposition,  $U_0$ is  the neighborhood obtained in  \cite[Lemma 6.4]{CJSwind}.
\end{proof} 
%; it is called {\em forward} (resp. {\em backward}) {\em complete} when only the upper (resp. lower) unboundedness of their intervals of definition  is required.  
The WRS is {\em  (geodesically) complete} when  its inextendible geodesics are defined on all $\R$.  The next result proves, in particular, that its extendibility as a geodesic becomes equivalent to continuous extendibility. 
\begin{prop}\label{rr_uff}
  Let $\gamma: [a,b)\rightarrow M$, 
$b<\infty$, be a   $\Sigma$-geodesic. If there 
exists a sequence $\{t_n\}\nearrow b$ such 
that $\{\gamma(t_n)\}_n$ converges to some 
$p\in M$, then $\gamma$ is extendible beyond 
$b$ as a   $\Sigma$-geodesic. 
\end{prop}
\begin{proof} Apply Proposition \ref{normalneigh} to $p=x_0$ to obtain the corresponding $\varepsilon>0$. Then, for some $n_0$  the length of $\gamma$ between $t_{n_0}$ and $b$ is smaller than $\varepsilon$, and Proposition \ref{normalneigh} can be claimed again to extend $\gamma$ beyond $b$. 
\end{proof}

%\begin{rem}\label{rr_uff2}
%{\bf \br \sout{In the case that we assumed that $\gamma$ is a wind curve but not necessarily a geodesic, the proof shows that $\tilde \gamma$ can be extended as a causal continuous curve and, then, $\gamma$ as a locally $H^1$ wind curve, according to Proposition \ref{rr_uff}.} \er }
%\end{rem}

Notice that, from the spacetime viewpoint, {\em the completeness of the WRS geodesics means that the they cross all the slices $t=$ constant of the SSTK spacetime.} This observation and Proposition \ref{bolas2} underlie the following result
%  Given $\Sigma$, its  {\em reverse WRS} is $-\Sigma= S_R-W$. Such a reverse structure is naturally assigned to the {\em past} lightlike direction of the SSTS spacetime \cite[Corollary 3.16]{CJSwind}.  Clearly, the reverse parametrization of  a geodesic for $\Sigma$ becomes a geodesic for $-\Sigma$ and $\Sigma$ is forward complete iff $-\Sigma$ is backward complete. From now on,  completeness will mean forward and backward completeness. We have the following characterization 
 \cite[Th.~5.9(iv)]{CJSwind}:
 
\begin{thm}\label{t_CauchyHyp}
The following assertions are
equivalent:
\begin{enumerate}[(i)]
\item%[(1)] 
A slice $S_t$ (and, then every slice) is a
spacelike Cauchy hypersurface.
\item%[(2)]
 All the c-balls $\hat B_{\Sigma}^+(x,r)$ and
$\hat B_{\Sigma}^-(x,r)$, $r>0$, $x\in M$, are compact.
\item%[(3)] 
All the (open) balls $B_{\Sigma}^+(x,r)$ and
$B_{\Sigma}^-(x,r)$, $r>0$, $x\in M$, are precompact.
\item%[(4)] 
  $\Sigma$ is geodesically complete.
\end{enumerate}
\end{thm}
It is also worth pointing out that  the other causal properties of the spacetime (as being globally hyperbolic even if the slices $S_t$ are not Cauchy) can  also be characterized in terms of tidy properties of the corresponding WRS, extending the results (1)---(4) for the stationary case (\S \ref{s222}).

\subsection{The extended conic Finsler metric $\bar F$} \label{s_3.2}
In \cite{JSpisa}, the properties of  distances associated with pseudo-Finsler metrics were studied in full generality. Indeed, the metrics  in that reference  were defined in an open conic  subset  without assuming that they can be extended to the boundary. 
The case of critical wind considered in \S \ref{s223} is included in this general case. Indeed, only the conic Finsler metric $F$ appears there and its extension to the boundary of $A$ in the set $TM\setminus\mathbf{0}$  would be naturally equal to infinity (this is the underlying  reason why this boundary did not contain allowed directions and, so,  it was not included in $A_E$). Bearing this infinite limit in mind, such an $F$ has been extensively studied in   \cite[\S 4]{CJSwind}.

However, as discussed in Remark \ref{r_dominios}, the metrics $F$ and $F_l$ are continuously extendible to the boundary  of $A$ in $TM_l\setminus\mathbf{0}$. %and $F$ even a bit beyond. 
Next, we will study 
%such extensions and, specially, 
 some elements associated with this extension of  $F$, %to $A_E\cap TM_l$, 
which will allow us to  extend Theorem \ref{t_CauchyHyp}, adding a further characterization to those appearing there. 

%here the case of WRS where the extension of the associated conic Finsler and Lorentz Finsler metrics to a part $A_E$ of its topological boundary becomes completely natural.

\subsubsection{$\bar F$-separation} 
Next, we will consider  the %{\em continuous} 
extension of $F$ explained in detail in Remark \ref{r_dominios}, in order to introduce a related {\em 
separation}, whose role has some similarities with a 
distance.
%from $A$ to its boundary in $TM_l$.
Even though the extension of $F$ has  already been taking into account in order to compute the length of wind curves, we will introduce explicitly the notation $\bar F$ in order to distinguish our study from previous ones, where the {\em $F$-separation} only takes into account the open domains (for example, in the  Randers-Kropina case studied in \cite[Section 4]{CJSwind} or in \cite{JSpisa}).
 
\begin{defi} The  {\it extended conic Finsler metric} of a WRS $(M,\Sigma)$ (or its associated SSTK splitting) is the map $\bar F: A\cup A_E \rightarrow [0,+\infty)$ given by \eqref{randers-kropina}
%Given a WRS $(M,\Sigma)$ with $\bar{F}$ the associated extended conic Finsler metric, we will denote by 
and the associated    {\em $\bar F$-separation} is the map $d_{\bar{F}}: M\times M\rightarrow [0,+\infty]$ given by:\footnote{ \label{f_duda}  For simplicity, we assume here that the  connecting curves   constitute the set of wind curves $C^\Sigma_{p,q}$. As   $\bar F$ is invariant under (positive) reparametrizations, one could also consider %(piecewise) smooth 
curves with velocities in $A\cup A_E$ %\br \sout{or} \er 
even vanishing in the region of mild wind  (the  relevant restriction for such a curve $\gamma$ would be the existence of  a parametrization  $\hat \gamma$  satisfying that $t\rightarrow (t,\hat\gamma(t))$ is causal),  but no more generality would be obtained.  }
\[d_{\bar F}(p,q)=\inf_{\gamma\in C^\Sigma_{p,q}} \ell_{\bar F}(\gamma)\]
for every $p,q\in M$. In particular,  $d_{\bar F}(p,q)=+\infty$ if and only if $C^\Sigma_{p,q}$ is empty.
\end{defi}
Notice that, in this definition, $\ell_{\bar F}(\gamma)$ is equal to the length $\ell_{F}(\gamma)$ in the definition of the wind balls, as we already used there the extension of $F$; however, the  $\bar F$-separation is different to the $F$-separation $d_F$ in  \cite[\S 4]{CJSwind} and \cite{JSpisa}  as, now, curves with velocity in $A_E$ are allowed, that is, $d_{\bar F}\leq d_F$.  
  Consequently, let us denote 
  \begin{equation}\label{bballs}
     B^+_{\bar{F}_p}(r)=\{v\in (A_E)_p\subset T_pM:  \bar F(v)<r\}, \qquad B^+_{\bar F}(p,r)=\{q\in M: d_{\bar F}(p,q)<r \},            
\end{equation}
where the latter is the {\em (forward)
 $d_{\bar F}$-ball  
of radius $r>0$ and center} 
$p\in M$;  consistently, 
$\bar B^+_{\bar F}(p,r)$ 
is its closure and dual {\em backward} notions appear replacing $d_{\bar F}(p,q)$ with $d_{\bar F}(q,p)$ in \eqref{bballs}. 
 
%\footnote{\br En fin, coincide con la de radio $\leq r$?se supone que para $r=0$ se define como el punto, consistemente con las wind balls?\er}.

For a wind Minkowski structure of strong wind on $\R^n$, $d_{\bar F}$ becomes discontinuous because it jumps from a finite and locally bounded  value of $d_{\bar F}(0,q)$ when $q$ belongs to $A\cup A_E$ to an infinite value when $q$ is outside.
However, $d_{\bar F}$ may be discontinuous even when it remains finite, resembling the behaviour of the time-separation (Lorentzian distance) on a spacetime.

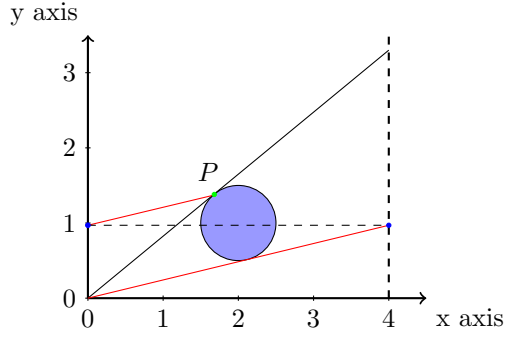
\begin{figure}
\begin{tikzpicture}
%\pgfpoint{0cm}{0cm};
\draw[thick,->] (0,0) -- (4.5,0) node[anchor=north west] {x axis};
\draw[thick,->] (0,0) -- (0,3.5) node[anchor=south east] {y axis};
\foreach \x in {0,1,2,3,4}
    \draw (\x cm,1pt) -- (\x cm,-1pt) node[anchor=north] {$\x$};
\foreach \y in {0,1,2,3}
    \draw (1pt,\y cm) -- (-1pt,\y cm) node[anchor=east] {$\y$};
\draw[thick,dashed] (4,0) -- (4,3.5);
\filldraw[fill=blue!40!white, draw=black] (2,1) circle (0.5cm);
\draw (0,0) -- (4,3.3);
\draw[red] (0,0) -- (4,0.97);
\draw[red] (0,0.97) -- (1.67,1.37);
\draw[dashed] (0,0.97) -- (4,0.97);
\filldraw[blue] (0,0.97) circle (0.8pt);
\filldraw[blue] (4,0.97) circle (0.8pt);
\filldraw[blue] (0,0.97) circle (0.8pt);
\filldraw[green] (1.68,1.38) circle (0.8pt);
 \node at (1.6,1.68)(s){$P$};

\end{tikzpicture}
\caption{\label{distance}  Wind Riemannian structure on a torus  $T^2 (= \R^2/\sim$, with  $(x,y)\sim (x',y')$ if and only if $(x-x')/4,(y-y')/4\in\Z $) with non-continuous $d_{\bar F}$  on points at finite separation.  }
\end{figure}
\begin{exe}
Let us consider the WRS induced  on the torus  $T^2=\R^2/4\Z $ from the wind Minkowskian structure on $\R^2$  whose indicatrix is  the sphere of radius $1/2$ centered at $(2,1)$, as depicted in Figure \ref{distance}. Then 
$d_{\bar F}((0,0),P)=1$ ($P$ as in the Figure \ref{distance}), but the points in the red line close to $P$ are at a distance much greater than 1. This concludes that the distance associated with $\bar F$ is not necessarily continuous even when it is finite and the WRS is geodesically complete.
\end{exe}
 Let us discuss some properties  related to {\em generalized distances} (compare with \cite[Section 3.1]{FHS}). The first one is very simple. 
\begin{prop}\label{p_3.9} If $\{p_n\}$ is a sequence in $M$ such that  $d_{\bar F}(p,p_n)\rightarrow 0$ or $d_{\bar F}(p_n,p)\rightarrow 0$  for some $p\in M$, then $p_n\rightarrow p$.
\end{prop}
\begin{proof}
 By Remark \ref{Riembounded}, there exists a Riemannian metric $h_R$ such that $\sqrt{h_R(v,v)}\leq \bar F(v)$ for every $v\in A\cup A_E$. As a consequence the $h_R$-distance from $p_n$ to $p$ is smaller than the $\bar F$-separation and it must go to 0, yielding the result. 
%We claim that there exists a Riemannian metric $h_R$ in $M$ 
% such that, at each $q\in M$ its indicatrix $S_q$ encloses $\Sigma_q$. 
%  Then, the $h_R$ length of any wind curve from $p_n$ to $p$ (or viceversa)  is smaller than its $\tilde F$-length. As a consequence the $h_R$-distance from $p_n$ to $p$ is smaller than the $\bar F$-separation and it must go to 0, yielding the result.
%
%To check the claim, start with any auxiliary Riemannian metric $\tilde h_R$. Now, at each point $q$ of the manifold $M$, there exists a minimum radius $r(q)>0$ such that $\Sigma_q$ lies inside the closed $\tilde h_R$-ball of the tangent space $T_qM$ centered at $0$ with radius $r(q)$. Such a radius varies continuosly with $q$, so, take any smooth function $R$ on $M$ such that  $R(p) > r(p)$ for all $p\in M$. The required metric is just $h_R= \tilde h_R/R^2$.
\end{proof}

However, as an important difference with the case of generalized distances, the converse may not hold. The following example (a small variation of 
the one introduced in the Figure 4 of \cite{JS}) will be useful for this and 
other purposes. %\footnote{\bb La pequenha modificacion 
%que sugeriria seria hacer que las regiones kropina no 
%fueran solo the lines $x=\pm 1, x=\pm 3$ sino  
%the strips 
%$(x_0-\epsilon, x_0+\epsilon)\times \R$ where $x_0=\pm 
%1, x_0=\pm 3$ Con eso seria obvia la infinitud de 
%$\{d_{\bar F}(p,p_n)\}$ 
%que digo abajo \eb }

\begin{exe}\label{ex_0} Consider the WRS $(g_R,W)$ in Figure \ref{RanKro}, where $g_R$ is just the usual Euclidean metric multiplied by a factor $1/5$, $W_{(x,y)}=f(x) \partial_x$ and the function $f$ behaves as depicted in the graph, so that the lines $x=\pm 1, x=\pm 3$ have critical wind. Clearly $\{p_n=(1+1/n,0)\}\rightarrow p=(1,0)$, but $d_{\bar F}(p_n,p)= \infty$ for all $n\in\N$ (recall, however, $\{d_{\bar F}(p,p_n)\}\rightarrow 0$).
\end{exe}

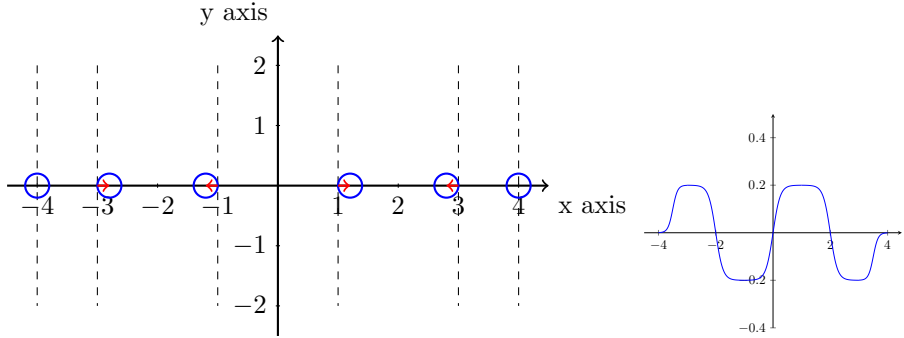
\begin{figure}
 \centering
\begin{tikzpicture}[scale=0.8]
\draw[thick,->] (-4.5,0) -- (4.5,0) node[anchor=north west] {x axis};
\draw[thick,->] (0,-2.5) -- (0,2.5) node[anchor=south east] {y axis};
%\draw (-3,0) .. controls (-2,1) and (-1,-1) .. (1,1);
%\draw [cyan, xshift=4cm] plot [smooth, tension=2] coordinates { (-3,0) (-2,1) (-1,-1) (1,1)};
\foreach \x in {-4,-3,-2,-1,1,2,3,4}
    \draw (\x cm,1pt) -- (\x cm,-1pt) node[anchor=north] {$\x$};
\foreach \y in  {-2,-1,1,2}
    \draw (1pt,\y cm) -- (-1pt,\y cm) node[anchor=east] {$\y$};
    \draw[dashed] (-4,2) -- (-4,-2);
    \draw[thick,blue] (-4,0) circle (0.2cm);
    \draw[dashed] (-3,2) -- (-3,-2);
    \draw[thick,blue] (-2.8,0) circle (0.2cm);
    \draw[thick,red,->] (-3,0) -- (-2.8,0);
    \draw[dashed] (-1,2) -- (-1,-2);
    \draw[thick,blue] (-1.2,0) circle (0.2cm);
    \draw[thick,red,->] (-1,0) -- (-1.2,0);
    \draw[dashed] (4,2) -- (4,-2);
    \draw[thick,blue] (4,0) circle (0.2cm);
    \draw[dashed] (3,2) -- (3,-2);
    \draw[thick,blue] (2.8,0) circle (0.2cm);
    \draw[thick,red,->] (3,0) -- (2.8,0);
    \draw[dashed] (1,2) -- (1,-2);
    \draw[thick,blue] (1.2,0) circle (0.2cm);
    \draw[thick,red,->] (1,0) -- (1.2,0);
    %\draw[blue] plot (\x,{\x});
   
\end{tikzpicture}
\begin{tikzpicture}[scale=0.5]
    \begin{axis}[
        domain=-4:4,
        xmin=-4.5, xmax=4.5,
        ymin=-0.4, ymax=0.5,
        samples=400,
        axis y line=center,
        axis x line=middle,
    ]%\addplot[domain=0:2, thin, samples=100]{7*x^4}
     \addplot+[domain=-4:-3,blue,mark=none,samples=100] {0.1*tanh(6*(x+3.5))+0.1};
     %\addplot+[domain=-3.05:-2.95,blue,mark=none,samples=10] {0.2};
     \addplot+[domain=-3:-1,blue,mark=none,samples=100] {0.2*tanh(-4*(x+2))};
      %\addplot+[domain=-1.05:-0.95,blue,mark=none,samples=10] {-0.2};
      \addplot+[domain=-1:1,blue,mark=none,samples=100] {0.2*tanh(4*(x))};
       \addplot+[domain=1:3,blue,mark=none,samples=100] {0.2*tanh(-4*(x-2))};
       % \addplot+[domain=-2.9:3,blue,mark=none] {0.2*sin(deg(0.5*3.1415*x))} node[above] {$f(x)$};
      % \addplot+[domain=2.95:3.05,blue,mark=none,samples=10] {-0.2};
        \addplot+[domain=3:4,blue,mark=none,samples=100] {-0.1*tanh(6*(-x+3.5))-0.1};
        %\addplot+[mark=none] {5/x};
    \end{axis}
\end{tikzpicture}
 \caption{\label{RanKro}  Even for a Randers-Kropina metric, the $\bar F$-separation may present subtle  properties, forbidden for a generalized distance (as  $\{p_n\}\rightarrow p$ but $d_F(p_n,p)=\infty$, see Example \ref{ex_0}), including ``black hole'' type behaviors  (Example \ref{ex324}).  }
 \end{figure}

With this caution, however,  notions related to  Cauchy sequences and completeness can be maintained.
 
% \begin{defi} A sequence $\{p_n\}\subset M$ is {\em forward Cauchy} (resp. {\em backward Cauchy}) for $\bar F$ if for all $\epsilon>0$ there exists $n_0\in \N$ such that $d_{\bar F}(p_n,p_m)<\epsilon$ (resp. $d_{\bar F}(p_m,p_n)<\epsilon$) whenever $n_0\leq n\leq m$.
%  $\bar F$ is {\em Cauchy complete} if all its forward and %backward Cauchy sequences are convergent.
% \end{defi}

\begin{defi}\label{defcomplete}
 Let $\bar F$ be the extended conic Finsler metric associated with a WRS  $(M,\Sigma)$ and $d_{\bar F}$ its associated separation:
\begin{enumerate}
\item  A subset  $\mathcal A\subset M$ is {\it forward} (resp. {\it backward}) {\it bounded}  if there exists $p\in M$ and $r>0$ such that $\mathcal A\subset B^+_{\bar F}(p,r)$ (resp. $\mathcal A\subset B^-_{\bar F}(p,r)$). 
\item  A sequence  $\{x_i\}$ is {\it forward} (resp. {\it  backward}) {\it  Cauchy}  if  for any $\varepsilon>0$ there exists $N>0$ such that $d_{\bar F}(x_i,x_j)<\varepsilon$ for all $i,j$ with $N<i<j$ (resp. $N<j<i$).
\item   The  space $(M,d_{\bar F})$ is {\it  forward} (resp. {\it backward}) {\it Cauchy complete }  if every  forward (resp. backward) Cauchy sequence converges to a point $q$. 
%such that \bb $\{d_{\bar F}(p_i,q)\}$ (resp. $\%{d_{\bar F}(q,p_i)\}$) \eb tends to zero. 
Moreover, $(M,d_{\bar F})$ is  {\em Cauchy complete}  if it is both, forward and backward Cauchy complete.
\end{enumerate}
\end{defi}

 It is worth noting that, even in the more restrictive framework of generalized distances, the convergence of all the forward Cauchy sequences does not imply the convergence of all the backward ones, so, to be only forward or backward Cauchy complete makes sense (see \cite[Section 3]{FHS} for a comprehensive study).

 In spite of  
 Example \ref{ex_0}, the following property holds for converging Cauchy sequences.
 
 \begin{prop} \label{r_ruf}
 If an $\bar F$-forward Cauchy sequence $\{p_n\}_n$ converges to a point $p\in M$, then 
 $\{d_{\bar F}(p_n,p)\}\rightarrow 0$.

 Moreover, if a forward  Cauchy sequence  $\{p_n\}_n$ admits a partial subsequence $\{p_{n_k}\}_k \rightarrow p\in M$ then all the sequence converges to $p$.
 \end{prop}
\begin{proof}
 For the first assertion, 
the Cauchy sequence $\{p_n\}_n$ will admit a 
 subsequence $\{p_{n_m}\}_m$ such that  
 $d_{\bar F}(p_{n_m},p_{n_{m+1}})<2^{-m}$. Then, by Proposition \ref{bolas2}, there exists a sequence $t_m$ such that $(t_{m+1},p_{n_{m+1}})\in J^+(t_m,p_{n_m})$ in the associated splitting $(\R\times M,g)$ and $t_{m+1}-t_m<2^{-m}$ for every $m\in \N$. Now define $\bar t=\lim_{m\rightarrow \infty} t_m$, which is finite (since $\bar t  -t_1  =\lim_{k\rightarrow \infty}\sum_{m=1}^k (t_{m+1}-t_m)<
 \sum_{m=1}^\infty 2^{-m} =2  $), consider a convex 
 neighborhood $V$ of $(\bar t,p)$  and let $N$ be big enough such that  $p_{n_m}\in V$ for $m\geq N$. As  $(t_m,p_{n_m})\in J^+(t_N,p_{n_N})$ 
 for $m\geq N$ and the causal relation is closed in any convex neighborhood (see \cite[Lemma 14.2]{Oneill83}), it follows that $(\bar t,p)\in J^+(t_N,p_{n_N})$. Moreover, for $m\geq N$, $d_{\bar F}(p_{n_m},p)\leq \bar t-t_m<\sum_{k=m}^\infty 2^{-m}\rightarrow 0$.  Using the triangle inequality and the definitions of $\bar F$-forward Cauchy sequences, we conclude that $\{d_{\bar F}(p_n,p)\} \rightarrow 0$.

 For the last assertion, recall that, then $\{d_{\bar F}(p_{n_k},p)\}_k \rightarrow 0$ and, by using  again  the triangle identity, $\{d_{\bar F}(p_n,p)\}\rightarrow 0$. So, convergence follows from Proposition~\ref{p_3.9}.
\end{proof}

 \subsubsection{$\bar F$-exponential  and geodesic balls } Next, our aim will be to define an exponential map  $\exp^{\bar F}_p$ for $\bar F$ at each point $p\in M$.
%[ the  defined in an open subset of ??? $(A_E)_p\subset T_pM$ and it can be extended to zero as $\exp^{\bar F}_p(0)=p$. ]
 This will extend the usual exponential for $F$ in $A$ (namely, constructed by using the formal Christoffel symbols  and the associated geodesics)  and will be well-defined and continuous even in the boundary of $(A_E)_p$ in the strong wind region (in this region,
 $\exp^{\bar F}_p$  will yield lightlike pregeodesics of $-h$). 
In order to ensure that both, $F$-geodesics and lightlike $h$-geodesics  match continuously we will work with pregeodesics of the SSTK spacetime and will project them on $M$.
However, recall that $F$ cannot be extended continuously to the zero section  in the critical region, so,   this section will be excluded first and the possibilities of extension will be studied specifically.

\begin{defi}\label{def} Let $\Sigma$ be a WRS on $M$. % and $p\in M$.
The {\em $\bar F-$exponential} %at $p$ 
is the map
$$
\exp^{\bar F}: \mathcal{U}\subset A \cup (A_E\setminus\mathbf{0}) \longrightarrow M \qquad \qquad v_p \mapsto \alpha(1)
$$ 
where, for each $v_p$, $\alpha$ is the geodesic   
constructed by taking the unique lightlike pregeodesic $\tilde\alpha$ in the associated SSTK splitting $(\R\times M,g)$ written as 
$$\tilde\alpha(s)=(\bar F(v_p) s,\alpha(s)), \qquad s\in [0,1],$$  (i.e., $\tilde\alpha$ is reparametrized  proportionally  to the projection $t:\R\times M\rightarrow \R$) and initial velocity 
$$\tilde\alpha'(0)=(\bar F(v_p),v_p) \in T_{(0,p)}(\R\times M),$$
 while $\mathcal{U}$ is the  open subset of $A \cup (A_E\setminus\mathbf{0})$ containing all the tangent vectors $v_p$ such that their associated pregeodesics $\tilde\alpha$ are defined on all $[0,1]$.
\end{defi}
Recall that the exact behaviour of all the $\Sigma$-geodesics as either 
 $F$, $F_l$ or boundary geodesics explained in Section \ref{s_313} is determined in \cite[Prop. 6.3]{CJSwind} (see also  Theorems 5.5 and 2.53 in that reference). Indeed, the curves $\alpha$ constructed in the previous definition correspond with the $F$-geodesics and the boundary geodesics for $\Sigma$. Bearing this in mind, the  following properties are in order.

\begin{prop}\label{p310}
%\begin{enumerate} \item 
For each $p\in M$ consider  the starshaped domain $\mathcal{U}_p:=T_pM\cap \mathcal{U}$ and the restriction $\exp^{\bar F}_p:=\exp^{\bar F}|_{\mathcal{U}_p}$.

\smallskip \noindent
%(1) \br If $\Lambda(p)\neq 0$, then  there exists a neighborhood $\tilde U_p$ of the zero vector $0_p\in T_pM$ such that $\tilde U_p \cap  
%(A \cup (A_E\setminus\mathbf{0}) \subset \mathcal{U}_p$. \er \footnote{\br Quitaria este punto , puesto que no se va a usar, y renumeraria los posteriores con cuidado por si se citan \er}

\smallskip \noindent
(1) If the wind on $p$ is mild (resp. critical; strong), then $\mathcal{U}_p \subset A_p$   
being $A_p=T_pM\setminus\{0_p\}$ (resp., $\mathcal{U}_p \subset A_p$, being $A_p$  an open half-space of $T_pM$;  $\mathcal{U}_p \subset (A_E)_p$, being $(A_E)_p$ a solid cone without vertex).

\smallskip \noindent
(2)  Assume that the curve $\alpha$ 
constructed for $v_p\in \mathcal{U}_p$ in Definition \ref{def} remains in:

(a) the region of 
mild wind (and, so, $\Sigma$ is the indicatrix of a Randers 
metric around the image of $\alpha$): then, $\alpha$ is an $F$-geodesic (so that $\exp^{\bar F}_p$ will agree with the natural $F$-exponential). 

(b) an open region of 
non-strong wind (and, so, $\Sigma$ is the indicatrix of a Randers-Kropina metric whose Christoffel symbols and geodesics can be computed as in the Finslerian case on all $A$) then, $\alpha$ is a  Randers-Kropina geodesic (according to \cite[Sect. 4]{CJSwind}). 

(c) the region $M_l$ of strong wind, (and, so, $\Sigma$ yields,  the indicatrix of both, a conic Finsler metric $F$ and a Lorentz-Finsler one $F_l$, plus the Lorentz metric $-h$): then $\alpha$ satisfies one of the following two exclusive possibilites: 

(ci) $\alpha'$ lies in the  open cones $A_l$ and $\alpha$ is an $F$-geodesic, or 

(cii) $\alpha'$ lies in the boundary $A_E\setminus A_l$ of the cones (thus, being a {\em boundary geodesic}) and it becomes a pregeodesic of the Lorentzian metric $-h$ on $M_l$.

\smallskip \noindent
(3) The $\bar F$-exponential  $\exp^{\bar F}$ is smooth on the open region $\mathcal{U}\cap A$ and continuous on all $\mathcal{U}$.

\smallskip 

\noindent
(4) The $\bar F$-exponential  $\exp^{\bar F}$ can be extended continuously to the zero section  away from the critical region. Moreover, the restriction of $\exp^{\bar F}$ to $\Sigma$  (or, in general, to the hypersurface $r \Sigma$ for any $r\geq 0$)  can be continuously extended to the zero section in the critical region.
\end{prop}

\begin{proof} 
%(1) This follows directly from 
% \cite[Prop. 6.5]{CJSwind} taking into account that, in its proof, the $F$-geodesics to be considered can be chosen not only those with initial velocity in $A_p$ but also in $(A_E)_p$.

(1) Straightforward from the definitions of the domains.

(2) See Th. 5.5 and 6.3 in \cite{CJSwind} (for part (cii), recall also Lemma 3.21); this matches with the summary in Section \ref{s_313} above. 
%\footnote{\bb En fin, me temo que me estoy pasando de repetitivo a estas alturas... \eb }.

(3)  Taking into account Definition \ref{def}, for each $v\in \mathcal{U}$ ($v\in T_pM$), put $\tau=\bar F(v)$,
%Consider the  associated SSTK spacetime $(\R\times M,g)$
%with natural projections $t, \pi$, and exponential map $\exp^g$ defined in some maximal open domain $\mathcal{D}\subset T(\R\times M)$  and observe that if $\alpha:[0,r]\rightarrow M$ is \bb any geodesic for $\bar F$ with speed $\tau=F(\alpha')$, 
consider the lightlike $g$-pregeodesic 
$\tilde\alpha(s)=(\tau s,\alpha(s))$ and reparametrize it as a geodesic   $\hat\alpha(h)=(t(h),x( h ))$  with the same initial velocity $(\tau,v)$. Then $(\tau s, \alpha(s))=(t(h(s)), x(h(s))$ for some function $h(s)$  with $h(0)=0, h'(0)=1$.  Therefore, 
$$ t'(h(s)) \; h'(s)= \tau ,  $$ 
that is, the function $h$  satisfies the ODE 
\begin{equation}\label{ODE}
 h'(s)=1/f(h(s),\tau,v),  \quad h(0)=0, \quad \hbox{where} \; \;  f(\bar s,\tau,v):= \frac{1}{\tau}  \left. \frac{d}{ds}\right|_{s=\bar s} t(\exp^g_{(0,p)}(s(\tau,v)))
 \end{equation}
   whenever $\exp^g_{(0,p)}$ is defined in $s(\tau,v)$. Notice that $f$ can be regarded as a smooth function on some maximal open subset of $\R\times \R_+ \times TM$ and, as $t$ is a temporal function, $f$ is strictly positive when applied on $(s,\tau,v)$ 
whenever $(\tau,v)$, regarded as a vector in $T_{(0,p)}  (\R   \times M)$, is 
 future-directed and causal. Summing up, for any $v=\alpha'(0) \in \mathcal{U}$, 
\begin{equation}\label{expdef}
\exp_p^{\bar F}(v)=\alpha(1)=x(\varphi(1,\bar F(v),v))=\pi(\exp^g_{(0,p)}(\varphi(1,\bar F(v),v)\cdot  (\bar F(v),v)),
\end{equation}
where $s\rightarrow \varphi(s,\tau,v)$ is the solution of \eqref{ODE}  which, obviously, depends smoothly on $(\tau,v)$. So, the expression \eqref{expdef} shows  that $\exp^{\bar F}$ is smooth everywhere  except at most in the boundary of $A_E$, because $\bar F$ is only continuous there.

(4) For the first assertion,  recall from \eqref{ODE} that the map $(s,v)\rightarrow f(s,F(v),v)$ is positive homogeneous of degree 0 in $v$, that is, $f(s,F(\lambda v),\lambda v)= f(s,F(v),v))$ for all $\lambda >0$. Then, $\varphi(1,\bar F(v),v)$ in \eqref{expdef} remains locally bounded outside the Kropina region (as we can take a compact neighborhood $W$ of $p$ which does not intersect the critical region,  and consider that $v$ varies in $\bar F^{-1}(1)\cap TW$, which is compact). So, when $v$ goes to $0$ the  variable  of the $g$-exponential in \eqref{expdef} goes to $0$ and $\exp_p^{\bar F}(v)$ goes to $p$, as required. 

%that, away from the critical region, $\bar F$ is continuous \bbr so that \eqref{expdef} concludes. \ebr \sout{, and its positive homogeneity  forces $r_1(\lambda v_p)\rightarrow 0$ when $\lambda\rightarrow 0$  in \eqref{e_s}.} 

For the second one, notice that $\bar F$ cannot be continuously extended to $0_p$ whenever $p$ lies in the critical region. However, as $\bar F$ is equal to one on $\Sigma\setminus\mathbf{0}$, it can be extended continuously to $1$ on $\Sigma$, and the continuity of the exponential in the SSTK spacetime ensures the result. Obviously, this can be extended to the case $r \Sigma$,   now  extending  $\bar F$   continuously  as  $\bar F(0)=r$. 
\end{proof}

\begin{rem}\label{r311}  (1)  It is easy to check that  $\exp^{\bar F}_p$ may be non-differentiable at the boundary $A_E\setminus A$, as the initial 
 data for $\gamma$ in Definition \ref{def} depends on $\bar F(v_p)$ and $\bar F$ may be non-smooth on the boundary.  Indeed, the root in the expression  \eqref{randers-kropina}
 becomes 0 there; in the region of strong wind, these zeroes are the lightlike vectors for the Lorentzian metric $-h$, which implies  non-smoothability there. 
 
 However, direct computations in a concrete example may be illustrative. Consider  as Zermelo data in $\R^2$ the usual scalar product and the wind vector $W=\partial_x+\partial_y$. The function 
 $\phi(s)=  \bar F(s\partial_x|_{(0,0)}+\partial_y|_{(0,0)})$ for $s\geq 0$ is smooth for $s>0$ but only continuous 
 for $s=0$. Indeed, for each $s\geq  0$,  there exists  a unique point $(x_s,y_s)$ in the convex part the indicatrix  $\Sigma_{(0,0)} = \{(x,y)\in T_{(0,0)}\R^2: (x-1)^2+(y-1)^2  =1  \}$ such that  $\phi(s)\cdot (x_s,y_s)=(s,1)$. Then, $\phi(s)=1/ y_s$, $x_s=s y_s$, substituting in the indicatrix $x_s$, necessarily $y_s=\left(s+1\pm \sqrt{2s}\right)/(s^2+1)$, the choice of its convex part selects the  positive sign for $ y_s$ and, then, $\phi(s)=(s^2+1)/\left(s+1 + \sqrt{2s}\right)$.

 (2) About the question of continuity in part  (4),  notice that, when the integral curve of $K$ through a point $p$ of  the critical region is a geodesic, then the continuous extension of $\exp^F|_{\Sigma}$ to $0_p$ yields simply the point $p$. However, this does not occur when such an integral curve is not a geodesic. So, in general, it is impossible to extend continuously the full exponential $\exp^{\bar F}_p$ to all the zero section. 
\end{rem}

Summing up, Proposition \ref{p310} (4) allows one to extend continuously the domain $\mathcal{U}$ of $\exp^{\bar F}$ in order to include the zero section away from the critical region. 
 What is more, the critical  region will not be an obstacle to define geodesic balls of  radius $r$, since one can extend the restriction of 
$\exp^{\bar F}$ to any $r\Sigma$.

\begin{defi} Let $\Sigma$ be a WRS on $M$, $p\in M$.
For any $r\geq 0$ such that $\bar B^+_{\bar{F}_p}(r)$
(recall the notation \eqref{bballs}) is included in the 
starshaped domain $\mathcal{U}_p \cup \{0_p\}$, the {\em  (forward) geodesic  $\bar F$-sphere of center $p$ and radius $r$} is the  set
$$
\mathcal{S}^+_{\bar F}(p,r)=\exp_p^{\bar F}(r \Sigma_p),
$$
where $\exp_p^{\bar F}$ is assumed to be extended to $0_p$, if necessary. Then, for any $r_0 >0$ such that 
$\mathcal{S}^+_{\bar F}(p,r)$ is defined for all $0\leq r<r_0$, the {\em (forward) geodesic $\bar F$-ball} and {\em closed (forward) geodesic $\bar F$-ball} of center $p$ and radius $r$ are, resp.,  the  sets
$$
\mathcal{B}^+_{\bar F}(p,r)=
\cup_{0\leq r<r_0} \mathcal{S}^+_{\bar F}(p,r), 
\qquad \bar{\mathcal{B}}^+_{\bar F}(p,r)= 
\hbox{closure}(\mathcal{B}^+_{\bar F}(p,r_0)).
$$
\end{defi}
\begin{rem}\label{rrr} Necessarily, $\mathcal{S}^+_{\bar F}(p,r)$ is compact and, whenever $\bar B^+_{\bar{F}_p}(r_0)\subset \mathcal{U}_p$ then $$\bar{\mathcal{B}}^+_{\bar F}(p,r_0)=\mathcal{B}^+_{\bar F}(p,r_0) \cup \mathcal{S}^+_{\bar F}(p,r_0)$$
and it is compact too (indeed, the right-hand side  is the projection of the compact subset in the SSTK splitting obtained by exponentianing the null vectors $w$ tangent to $(0,p)$ such that $0\leq dt(w)\leq r_0$ with a $t$-reparametrized pregeodesic  analogously as in the proof of part (3) in Proposition \ref{p310}). 
 However, $\mathcal{B}^+_{\bar F}(p,r)$ is not necessarily open; in fact, this happens even for wind Minkowskian structures, as $A_p\cup (A_E)_p$ is not open away from the mild wind region.\end{rem}

\begin{defi}
A curve $\gamma: [a,b]\rightarrow M$, $p=\gamma(a), v_p=\gamma'(a)$ is an {\em  $\bar F$-geodesic} of a WRS $\Sigma$ if either  $v_p\in A_p\cup((A_E)_p\setminus{0_p})$ and   $\gamma(s)=\exp^{\bar F}_p((s-a)v_p)$ for all $s\in[a,b]$ or $\gamma'(a)=0_p$, $\Lambda(p)=0$ and there exists some $c>0$ such that the curve $[a,b]\ni s \mapsto (c(s-a),\gamma(s))\in \R\times M$ is  a lightlike pregeodesic  of the SSTK spacetime  with initial condition $c\partial_t|_{(0,p)}$ at the instant $a$.   
\end{defi}
The previous definition extends naturally to non-compact intervals. Namely, for the case $[a,b)$  no modification is necessary, while for the case $(a,b]$ one assumes that, for all $t\in (a,b)$, the restriction $\gamma|_{[t,b]}$ is a geodesic. Indeed, {\em  the $\bar F$-geodesics  are the  $F$-geodesics, boundary geodesics and exceptional geodesics} explained in Section~\ref{s_313} ({\em i.e. all the $\Sigma$-geodesics except the $F_l$ ones}).  The lack of symmetry of $\Sigma$ makes meaningful the following distinction, as in the case of classical Finsler metrics. 

\begin{defi}\label{d316}  $\bar F$ is {\em forward} (resp. {\em backward}) {\em complete} if the domain of all its inextendible $\bar F$-geodesics is upper (resp. lower) unbounded. 
\end{defi}
 Clearly, given a WRS $\Sigma$,  $\bar F$ is forward complete if and only if the extended conic Finsler metric for the reverse WRS $\tilde \Sigma= -\Sigma$ is backward complete. 

\subsection{Main result on the completeness of $\Sigma$} \label{s_3.3}
In order to obtain our main result, let us start  strengthening the relations between $\bar F$-balls and geodesics. 
   %Moreover, as $d$ is continuous away from the diagonal, the closure
 %$\bar{\mathcal B}^+_p(r)=\{ x\in M: d(p,x)\leq r\}\cup \{p\}$.
\begin{lemma}\label{lem:exp}
Let $(M,\Sigma)$ be a WRS and $x_0\in M$. For each neighborhood $W_0$ of $x_0$ there exists  another neighborhood $U_0\subset W_0$ and some $\varepsilon>0$  such that, for every $x\in U_0$ and $0<r< \varepsilon$: 
\begin{enumerate}[(i)]
\item Both, the c-balls $\hat{B}^+_{\Sigma}(x,r)$ and the closed $\bar F$-balls $\bar{B}^+_{\bar F}(x,r)$    are  compact  and included in $W_0$.  
%\item The \bb non-constant \eb $\bar F$-geodesics departing from $x\in U_0$ of length less than $\varepsilon$ are  minimizing. \br minimizing para la $\bar F$ separacion? es esto cierto?? \er %, and
\item  
%When $x$ does not belong to the critical region, \eb then   
$\bar B^+_{\bar F_x}(r)\subset T_xM$ is included in  $\mathcal{U}_x \cup \{0_x\}$   and, then, the geodesic and metric $\bar F$-balls coincide, i.e.: 
$\mathcal{B}^+_{\bar F}(x,r) = B^+_{\bar F}(x,r)$ and $    
\bar{\mathcal{B}}^+_{\bar F}(x,r) =\bar B^
+_{\bar F}(x,r).$

%\noindent Otherwise, \bb $\bar B^+_{\bar F_x}(r)%\setminus \{0_x\}$ \eb is included in $\mathcal{U}_x$ and let

%$G_x(r)=\pi(\hbox{{\rm Im}}(\rho_x)\cap t^{-1}([0,r))), \bar G_x(r)=\pi(\hbox{{\rm Im}}(\rho_x)\cap t^{-1}([0,r])),$
 
%\noindent be the projections on $M$ of the pieces of the SSTK-geodesic $\rho_x$ with $\rho'_x(0)=\partial_{(0,x)}$ contained in, resp.,  $[0,r)\times M$ and $[0,r]\times M$. Then:

 %$\exp^{\bar F}_x(B^+_{\bar F_x}(r) \bb \setminus \{0_x\}) \cup  G_x(r) \eb =\bar B^+_{\bar F}(x,r), $ and
 
%$\exp^{\bar F}_x(\bar B^+_{\bar F_x}(r) \bb \setminus \{0_x\}) \cup \bar G_x(r) \eb =\bar B^+_{\bar F}(x,r)$. %and $\exp_x(S_x(r))={\mathcal S}^+_x(r)$, $\bar{ B}^+_F(x,r)$ %and ${\mathcal S}^+_x(r)\cup \{p\}$ are

\end{enumerate}
\end{lemma}
\begin{proof} 
%plus the compactness of the closed $\Sigma$ balls $\bar B^+_\Sigma(x,r)$ (being also included in $W_0$) 
Roughly, the result follows  from \cite[Lemma 6.4]{CJSwind} and Prop. \ref{normalneigh}. Indeed, \cite[Lemma 6.4]{CJSwind} provides directly both, the neighborhood $U_0$ and $\varepsilon>0$ such that the c-balls $\hat{B}^+_{\Sigma}(x,r)$ are compact  for every  $x\in U_0$ and   $0<r<\varepsilon$.  Even more, this also proves that $\cup_{r'\in [0,r]}\hat{B}^+_{\Sigma}(x,r')$  (which is equal to the projection on $M$ of the compact set $J^+(0,x)\cap t^{-1}([0,r])$, recall Proposition \ref{bolas2}  and the choice of $U_0$) is compact too. As this set includes 
${B}^+_{\bar F}(x,r)$, this proves the compactness of its closure, concluding $(i)$.

%The minimizing property stated in (ii) is ensured in \cite[Prop. 6.5]{CJSwind} (recall that, even though such a property is stated for $F$-geodesics with initial velocity in $A_{x_0}$, the proof applies directly when they lie in $(A_E)_{x_0}$ as well as when they start at the other points $x\in U_0$). 

 For $(ii)$  we claim first  that if $U_0$ is  obtained as above then 
 $[0,\varepsilon]\times U_0$ can be assumed to lie in a 
 globally hyperbolic (and thus, causally simple) 
 neighborhood $U$ of the SSTK splitting satisfying $\pi(U)
 \subset W_0$. Indeed,  $(0,x_0)$ admits an arbitrarily 
 small globally hyperbolic neighborhood $U$ (see \cite[Th. 
 2.14]{MinSan}), and we have just to assume that $\pi(U)
 \subset W_0$. Then, there exists a small neighborhood 
 $W'_0\subset \pi(U)$ of $x_0$ and some $\varepsilon'>0$ 
 such that $[0,\varepsilon']\times W'_0\subset U$. So, the 
 claimed property follows just by repeating the step $(i)$  
 imposing $U_0\subset W_0' (\subset W_0)$ and choosing $
 \varepsilon<\varepsilon'$.

 Proposition \ref{normalneigh} implies that $\bar B^+_{\bar F_x}
(r)\setminus \{0_x\} \subset \mathcal{U}_x$ for every $x\in U_0$.  Even more, for the required equalities, the inclusion
$\subseteq$ follows trivially, as  all the points in the geodesic ball are reached by a geodesic of $\bar F$-
length smaller  than the radius (or equal to it in the closed case, see Remark 
\ref{rrr}), which can be regarded as a wind curve (indeed, so is the $\bar F$-geodesic whenever $r\leq 1$;  otherwise, it can be reparametrized affinely as a wind curve). For the inclusion $\supseteq$ in the first equality, recall that for each $y\in B^+_{\bar F}(x,r)$ the connecting wind curve $\gamma: [0,r']\rightarrow M$, $r'<r$ yields a causal curve $[0,r']\ni t\mapsto (t,\gamma(t))$. Then $z'=(r',\gamma(r'))\in J^+(0,x)$ and  being $[0,\varepsilon]\times U_0$ included in the causally simple subset $U$,  there exists a first point $z\in  [0,\varepsilon]\times W_0$  on the integral curve of $\partial_t|_{z'}$ which belongs to $J^+(0,x)$. Then, the unique lightlike geodesic from $(0,x)$ to $z$ projects into the required $\bar F$-geodesic. Moreover, then the inclusion $\supseteq$ in the second equality also follows just applying the compactness  (and then closedness)  of $\bar{\mathcal{B}}^+_{\bar F}(p,r_0)$,  see Remark \ref{rrr}. 
\end{proof}

\begin{prop}\label{prop:connect}
Let $(M,\Sigma)$ be a WRS and $\bar F$ the associated extended conic Finsler metric. Given  $p\in M$ assume that $\exp^{\bar F}_p$ (resp.  exponential at $p$ for the reverse WRS $\tilde \Sigma=-\Sigma$) is defined in the whole domain $(A_E)_p\subset T_pM$ (resp. $-(A_E)_p\subset T_pM$) of $F_p$. Then, for any $q\in M$, $ q\neq p$, such that $d_{\bar F}(p,q)$ (resp. $d_{\bar F}(q,p)$)  is finite, there exists a minimizing $\bar{F}$-geodesic from $p$ to $q$ (resp. from $q$ to $p$).
\end{prop}
\begin{proof}
Consider a sequence of  wind  curves $\alpha_n$ from $p$ to $q$ 
such that $\lim_n \ell_{\bar F}(\alpha_n)=d_{\bar F}(p,q)$. Let  $\tilde\alpha_n(t)=(t,\alpha_n(t))$  be the graph of $\alpha_n$ in the associated SSTK spacetime $(\R\times M,g)$.
%, namely, $\rho_n(s)=(s,\alpha_n(s))$. 
 As  the curves $\tilde\alpha_n$ are  future-directed causal, then  there exists a limit curve $\tilde\alpha(s)= (s,\alpha(s))$  starting at  $(0,p)$ defined in a subinterval  $I$  of $[0,d_{\bar F}(p,q)]$,  $0\in I$,  such that $\lim_n \alpha_n(s)=\alpha(s)$  for all $s\in I$  (see  \cite[Sect. 3.3]{BEE} or  \cite[Lemma 5.7]{CJSwind}). This  implies that $\tilde\alpha$ is a lightlike pregeodesic  and
\begin{equation}\label{longitud}
d_{\bar F}(p,\alpha(s_0))=\ell_{\bar F}(\alpha|_{[0,s_0]}),
\end{equation}
for all $s_0\in I\setminus\{0\}$. 
 Indeed, if  either  $\tilde\alpha$ is not a lightlike pregeodesic  or \eqref{longitud} does not hold (only the inequality $<$ should be taken into account then), we claim that there exists $\delta>0$ such that $(s_0-\delta,\alpha(s_0))\in J^+(0,p)$. This follows from Proposition \ref{bolas2} if \eqref{longitud} does not hold. When $\tilde\alpha$ is not a pregeodesic, observe  that
  $(s_0,\alpha(s_0)) \in I^+(0,p)$ 
 (see \cite[Proposition 10.46]
 {Oneill83})  and, as  the chronological relation is open,  there exists $\delta>0$ such that $(s_0-
 \delta,\alpha(s_0))\in I^+(0,p)\subseteq J^+(0,p)$.   So, for the claimed $\delta$, 
 %assume  that there exists $\delta>0$ such that $(s_0-\delta,\alpha(s_0))\in J^+(0,p)$.  Then 
 there exists a future-directed causal curve $\tilde\rho$ from $(0,p)$ to $(s_0-\delta,\alpha(s_0))$. Concatenating this curve with $[s_0-\delta,\ell_{\bar F}(\alpha_n)-\delta]\ni s\rightarrow (s-\delta,\alpha_n(s))\in \R\times M$, we obtain  future-directed causal 
 curves   $\tilde\gamma_n$  from $(0,p)$ to $(\ell_{\bar F}(\alpha_n)-
 \delta,q)$. Taking $n$ big enough, one gets  that the projection of $\tilde\gamma_n$ is a curve from $p$ to $q$ which has length equal to  $
 \ell_{\bar F}(\alpha_n)-\delta<d_{\bar F}(p,q)$ 
   in 
 contradiction with   the definition of $d_{\bar F}(p,q)$.  Therefore, $\tilde\alpha$ is a lightlike pregeodesic and \eqref{longitud} holds,  as required.

Now, the discussion in Section~\ref{s_313} %\cite[Theorem 5.5]{CJSwind} 
   implies that,  being $\tilde\alpha$  a lightlike pregeodesic, necessarily $\alpha$ is either  a unit  geodesic   for $F$ or $F_l$, or a boundary geodesic  (exceptional geodesics are excluded as $p\neq q$). 
    Moreover, if $\alpha$ were an $F_l$-geodesic which is not a boundary one,  then it would be included in $M_l$, $\alpha'$ could not vanish and $0<F(\alpha')<F_l(\alpha')\equiv 1$  (recall that $F\leq F_l$ holds always and, if the equality occurred in our case, then $h(\alpha',\alpha')=0$ at some point, that is, $\alpha$ would be the boundary geodesic corresponding  to the  initial  velocity at that point). Then, 
%$$\ell_{\bar F}(\alpha)=\int_0^{d_{\bar F}(p,q)}F(\alpha'(s))ds< d_{\bar F}(p,q),$$
$$\ell_{\bar F}(\alpha|_{[0,s_0]})=\int_0^{s_0}F(\alpha'(s))ds< d_{\bar F}(p,\alpha(s_0)).$$
As this is a contradiction with \eqref{longitud}, $\alpha$ becomes  either a boundary or $F$-geodesic and, thus, an $\bar F$-geodesic. So, by the hypothesis on  $\exp^{\bar F}_p$, $\alpha$ is defined in $[0,d_{\bar F}(p,q)]$, which concludes the proof. 
\end{proof} 
 Finally, recalling Definitions \ref{defcomplete} and \ref{d316}, we can state our main result. 

\begin{thm}\label{HopfRinow}
Let $(M,\bar F)$ be the conic Finsler manifold associated with a WRS $(M,\Sigma)$. Then $(M,\Sigma)$ is geodesically complete if and only if $(M,\bar F)$ is  geodesically  complete. In addition, the following conditions are equivalent:
\begin{enumerate}%[(a)]
\item[(a)] The  space $(M,d_{\bar F})$ is forward (resp. backward) Cauchy complete.
\item[(b)] $(M,\bar F)$ is forward (resp. backward) geodesically complete.
%\item[(c)] At every point $p\in M$, $\exp_p$ (resp. backward $\exp_p$) is defined on all of $A_p\subset T_pM$.
\item[(c)] Every closed and forward (resp. backward) bounded subset of $(M,d_{\bar F})$ is compact.
\end{enumerate}
Moreover, any of the above conditions implies that $\bar F_l$ is forward (resp. backward) geodesically complete.

Finally, if $p_0\in M$ has the following property: $d_{\bar F}(p_0,q)$ (resp. $d_{\bar F}(q,p_0)$) is finite for every $q\in M$, then the above conditions are equivalent to 
\begin{enumerate}
\item[(d)] At   $p_0\in M$,  $\exp^{\bar F}_{p_0}$ 
(resp. backward $\exp^{\bar F}_{p_0}$) is defined on all 
$(A  \cup A_E  )_{p_0}$  (resp. $-(A  \cup A_E  )_{p_0}$). 
\end{enumerate}
\end{thm}
\begin{proof} 
Let us start with the equivalences among the displayed items. The introduced framework will allow us to use  
standard  arguments  as in \cite[Th. 6.61]{BaChSh00}.
for $(a)\Rightarrow (b) \Rightarrow (c)$ 
and $(d)\Rightarrow (c)$.  
We will reason always for the forward case. 

%\smallskip

%\noindent
$(a)\Rightarrow (b)$. Otherwise take an  
 incomplete geodesic $\gamma: [0,b)\rightarrow M, b<\infty$. As the sequence $\{\gamma (b-1/m)\}_{m>1/b}$ is forward Cauchy, then it will have a limit $p$, obtaining so a contradiction (recall Proposition \ref{rr_uff}).  

%\smallskip

%\noindent
$(b)\Rightarrow (c)$. Let $\{p_m\}_m$ be any forward bounded sequence and let $p\in M, r>0$ such that $\{p_m\}_m\subset B^+_{\bar F}(p,r)$. By hypothesis, Proposition \ref{prop:connect} is applicable and, thus, $\{p_m\}_m$ lies in the geodesic ball $\mathcal{B}^+_{\bar F}(p,r)$. Then, the existence of a converging partial subsequence of $\{p_m\}_m$ follows because  the closure of this ball is compact (recall  Remark \ref{rrr}). 

%\noindent
$(d)\Rightarrow (c)$. As in the previous case, choosing now $p=p_0$ (recall that  $B^+_{\bar F}(p,r)\subset B^+_{\bar F}(p_0,r+d_{\bar F}(p_0,p))$ and $d_{\bar F}(p_0,p)$ is finite). 

%\smallskip
%\noindent 
$(c)\Rightarrow (a)$.  Let $\{p_m\}_m$ be a forward Cauchy sequence. %(the case for backward Cauchy sequences is analogous). 
The triangle inequality for $d_{\bar F}$ 
implies that $\{p_m\}_m$ is forward bounded. 
 So, its closure is compact and $\{p_m\}_m$ 
%\footnote{Armonizar indices y 
%sucesiones} 
admits a converging subsequence. Therefore, 
 the result follows by Proposition 
\ref{r_ruf}.  

\smallskip
\noindent
 For the  statement  about the geodesic 
completeness of $F_l$, observe that $\ell_{\bar F}
(\alpha)\leq \ell_{\bar F_l} (\alpha)$. This easily 
implies that if $\gamma:[0,b)\rightarrow M$,  $b<
\infty$,    is an $F_l$-geodesic, then  the 
sequence $\{\gamma(t_m)\}_m$ with $\{t_m=b-1/m\}$ 
%with $t_i<t_j$ for $i<j$ and $\lim_i t_i=b$, the %sequence  
is forward Cauchy for $d_{\bar F}$. Thus,  $\gamma$ 
is extendible to $b$  as an $F_l$-geodesic (Proposition \ref{rr_uff}), as required. 

 Finally, the first statement follows because the $\Sigma$-geodesics are the geodesics of both $\bar F$ and $F_l$ and we have just proved that the completeness of $\bar F$-geodesics implies the completeness of $F_l$-geodesics.    
\end{proof}

\begin{exe} \label{ex324}
In the last theorem,   the finiteness of $d(p_0,q)$ (or $d(q,p_0)$)  for every $q\in M$ is necessary  to obtain  the equivalence between $(d)$ and the other properties.  In fact, if we consider $M=\R^n\setminus\{(0,0,\ldots,0,-1)\}$, $g_R$ the Euclidean metric and $W=(0,\ldots,0,1)$, the corresponding Kropina metric satisfies that the  forward exponential  map  at  ${\bf 0}=(0,0,\ldots,0)$ is defined in the maximal domain, but the associated distance is not forward complete.  Notice, however, that
%Moreover, given $q=(x_1,x_2,\ldots,x_{n-1},0)$, we have that 
$d_F({\bf 0},q)=+\infty$ whenever $q=(x_1,\ldots,x_n)$ satisfies $x_n\leq 0$. 

 As a more sophisticated example, recall  Example \ref{ex_0} (Figure \ref{RanKro}). 
%, where we consider the Euclidean space with a factor $1/5$ and $W_{(x,y)}=f(x) \partial_x$. 
The regions $x\geq 4$ and $x\leq -4$ are Euclidean but  cannot be connected by any wind curve.  So, the exponential at any point in the region $x\geq 4$ will be defined on all its tangent space, even if one removes a point of the region $x\leq -4$ (making incomplete the WRS).

 Recall also that the regions $-3<x<-1$ and $1<x<3$ behave as a sort of ``black holes'', namely,  once you enter there, it is not possible to go out. 
%, and then they are disconnected from the rest of the manifold.
\end{exe}

\section{Some applications for SSTK spacetimes}\label{s_4}

\subsection{Cauchy hypersurfaces in SSTK}\label{s_4.1}
As a direct consequence of Theorems \ref{t_CauchyHyp} and \ref{HopfRinow}, a characterization of Cauchy hypersurfaces is obtained.
\begin{cor} Let $(\R\times M,g)$ be an SSTK splitting. The following assertions are
equivalent:
\begin{enumerate}\item The slices $t=$ constant are Cauchy hypersurfaces.
\item The associated WRS, $\Sigma$, is  geodesically  complete.
\item The extended conic Finsler metric $\bar F$ of  $\Sigma$ is complete (in any of the equivalent senses of Theorem \ref{HopfRinow}).
\end{enumerate}
\end{cor}

This precise characterization of WRS-completeness/Cauchy slices may be useful for concrete examples. Indeed, incompleteness would follow just by finding an incomplete Cauchy sequence or geodesic for $\bar F$.  Next,  some criteria to ensure completeness are obtained.

\begin{prop}\label{p_4.2}  Let $\Sigma$ be a WRS on $M$ with associated conic Finsler metric\footnote{Notice that  either using the conic Finsler metric $F$ 
or the extended one $\bar F$  in the statement of this result are equivalent.  Indeed,  the inequality \eqref{eH}  holds when  $F$ is replaced by $\bar F$, since $\bar F$ is continuous everywhere  except at most in the zeroes of the critical region, where the inequality holds trivially.} 
$F$, and let $H$ be an auxiliary  Finsler metric such that
\begin{equation}
\label{eH}
H(v)\leq F(v), \qquad \qquad \forall v\in A .
\end{equation}
If $H$ is  forward (resp. backward)  complete then $\Sigma$ is forward (resp. backward) geodesically  complete.
\end{prop}
\begin{proof} The usual generalized distance $d_H$ associated with $H$ satisfies that $d_H \leq d_{\bar F}$. Therefore, any (forward or backward) Cauchy sequence for $\bar F$ will also be Cauchy for $H$ and, as $H$ is 
complete, then it will  converge to some point of $M$.
\end{proof}
 Recall that the inequality \eqref{eH} means that the  indicatrix $\Sigma_H$ encloses the indicatrix $\Sigma$ at each $p\in M$.  In order to apply  the previous criterion sharply, the indicatrix $\Sigma_H$  should fit as much as possible in $\Sigma$. However, Riemannian metrics are easier to  handle  in practice and this may be enough in some particular cases.  First criteria are the following. 

\begin{prop}\label{p_4.3}
Let $\Sigma$ be a WRS with Zermelo data $(g_R,W)$, and let $W^\flat=g_R(W,\cdot)$.   Then, $\Sigma$ is complete if one of the following conditions holds:
\begin{enumerate}
\item[(i)] the conformal metric $h^*=g_R/(1+|W|_R)^2$ is complete, or
\item[(ii)] the metric $g_R$ is complete and $|W|_R$ grows at most linearly with the $d_R$-distance, that is, there exist $\lambda_0, \lambda_1>0$, $x_0\in M$:
$$
|W_x|_R \leq \lambda_0+\lambda_1 \; d_R(x_0,x) \qquad \qquad \forall x\in M.
 $$
\end{enumerate}
\end{prop}
\begin{proof} For $(i)$, just notice  that, at each $p\in M$, the indicatrix of $h^*$ is a $g_R$-sphere of radius $|W|_R+1$.   So, it contains $\Sigma_p$ and Prop. \ref{p_4.2} can be applied. 

The conditions in $(ii)$ imply $(i)$. Indeed the metric  $h^*$  is conformal now to the complete one $g_R$ with a conformal factor $\Omega=1/(1+|W|_R)^2$  which decreases at most quadratically with the distance. So, it is well-known that $h^*$ is then complete (namely, if $\gamma: [0,\infty)\rightarrow M$ is any diverging curve parametrized with unit $g_R$-velocity, then its $h^*$-length satisfies 
$\int_0^\infty h^*(\gamma'(s), \gamma'(s))^{1/2}ds\geq 
\int_0^\infty \frac{ds}{1+\lambda_0+\lambda_1 s}=\infty$.
\end{proof}

%for some $A, \epsilon>0$: 
%\int_0^\infty \epsilon \; ds/(A+s) =\infty$).\footnote{\br %A mi me sale $\int_0^\infty h^*(\gamma'(s), \gamma'(s))^{1/2}ds\geq 
%, lo sustituimos por la expresion en $A$ y $\epsilon$? \er 

 A more accurate  consequence of Proposition \ref{p_4.2}   is the following.
\begin{prop}\label{c_4.3}
 A  WRS $\Sigma$  with Zermelo data $(g_R,W)$  is complete if so is  the Riemannian metric 
$$
h =  g_R - 
\frac{1}{1+|W|_R^2} W^\flat \otimes W^\flat ,
$$
 where $W^\flat$ is computed with $g_R$, namely, $W^\flat(v)=g_R(v,W)$ for all $v\in TM$.  
\end{prop}

\begin{proof} 
Taking into account that the completeness of $h$ and $h/2$ are equivalent, 
 we have just to check that
$ h(W+U,W+U)\leq 2$, where $U$ is any $g_R$-unit vector field:
$$\begin{array}{l}
h(W+U,W+U)  =  |W|_R^2+1+2g_R(W,U)-\frac{\left( |W|_R^2+g_R(W,U)\right)^2}{1+|W|_R^2}\\
 = \frac{1}{1+|W|_R^2}
\left(\left(1+|W|_R^2\right)^2+ 2  (1+|W|_R^2) \, g_R(W,U)
-\left(|W|_R^2+g_R(W,U)\right)^2\right)\\
=\frac{1}{1+|W|_R^2}
\left(1+2|W|_R^2 + 2 g_R(W,U)-g_R(W,U)^2\right) ,
\end{array}
$$
where the last parenthesis can be regarded as a quadratic polynomial in $g_R(W,U)$. 
%\br \in [-|W|_R,|W|_R]\er$.\footnote{\br Yo quitaria el intervalo, pues al final no es necesario, ya que el maximo en $x=1$ es absoluto y $1$ podria no estar en ese intervalo.\er} 
This takes its maximum when $g_R(W,U)=1$, yielding so the  required inequality.  \end{proof}

\begin{rem}  One can also take  other choices of Riemannian metric, 
which are not conformal to $g_R$ but may be better adapted 
to the shape of $\Sigma$. For example, given $\Sigma $ as 
in the proposition above, it is not hard  to check that $\Sigma$ is enclosed by the indicatrix of the following metric   
\begin{equation} \label{ehl}
h_\lambda = \frac{1}{\lambda^2} \left(g_R - 
\frac{1}{\lambda^2-1+|W|_R^2} W^\flat \otimes W^\flat\right),
\end{equation}
where $\lambda: M\rightarrow \R$ is any function satisfying $\lambda>1$ (indeed, the metric $h$ in 
Proposition \ref{c_4.3} is just $h=\lambda^2h_{\lambda}$ for $\lambda=\sqrt{2}$). Therefore, the completeness of $h_\lambda$ for such a function  implies that $\Sigma$ is complete too. 
Notice that, when $\lambda$ is close to  $1$, the metric $h_\lambda$ is very close to $g_R$ in the directions orthogonal to $W$ (this may be an advantage)
but not in the direction of  $W$, as $h_\lambda(W,W)$ becomes very small (this may be a disadvantage); the situation is the other  way around for big $\lambda$.
%; so, depending on the case, either $h^*$ or $h_\lambda$ may be preferred.  
\end{rem}

%Recall that the easiest  the application of the criterion, the weakest  the result. Indeed, the criterion (ii) implies (i) (as seen explicitly in the proof), and the latter implies the completeness of  $g_R$ (as $g_R\geq h^*$); however, Proposition \ref{p_4.2} can be applied even when $g_R$ is not complete.

%\begin{rem} 
Recall that the  easier   the application of the previous criteria, the  weaker   the result. Indeed,  in Proposition \ref{p_4.3},  the criterion $(ii)$ implies $(i)$ (as seen explicitly in the proof), the latter implies the completeness of $h$ (as $h^* \leq h$),  
and the completeness of $h$ also implies the completeness of $g_R$ (as $g_R\geq h$); however, Proposition \ref{p_4.2} can be applied even when $g_R$ is not complete. The next examples illustrate these possibilities.

%\end{rem}

\begin{exe} 
{\em Sharpness of the rough bounds.}   Let  $(M,g_R)= \R^2$, $W_{(x,y)}=f(x,y)\partial_x$, for some smooth function $f$ on $\R^2$.
 
 {\em Bound (ii)}. Choose $r\in \R$ and put
  $f(x ,y)=|x|^r$ whenever $|x|\geq 1$. When $r\leq 1$, the growth of $W$ is at most linear and thus, the corresponding WRS is complete. However, if $r>1$ the WRS is incomplete. Indeed, the inextendible  curve $[0,L) \ni s\mapsto (x(s),0)$ with $x(0)=1$ and  $x'(s)=1+x^r(s)$ diverges (it escapes from any compact subset), it is $\bar F$-unit and has finite length,  $L=\int_1^\infty dx/(1+  x^r)  <\infty$. 
  
  {\em Bound (i)}. Choose $f \equiv 0$ except in the squares $(n-1/n^4, n+1/n^4)\times (-1,1)$ where $|f|$  reaches the maximum $n^2$. Now,  $(i)$ is fulfilled but not $(ii)$.   
 
 {\em Bound with $h$}.  Put $f(x,y)=x e^y$. The metric $h^*$ is not complete (say, the curve $[0,\infty)\ni s\mapsto  (1,s)$
 %\footnote{\br estaba escrito $(0,s)$ pero yo diria que esa tiene longitud infinita \er} 
 has finite length), but $h$
(as well as $h_\lambda$ for any constant $\lambda>1$) is complete. Indeed,
$$ h=\frac{1}{1+x^2 e^{2y}}dx^2+dy^2,  $$
  thus, if $\gamma(s)=(x(s),y(s))$ is a diverging curve and $y$ is unbounded (resp. $|y|$ is bounded by some  $C>0$),  then its length is infinite because it is lower bounded  by, say,  $\int |y'(s)| \, ds$ (resp. $(\int  (|x'(s)| \, /\sqrt{1+x(s)^2e^{2C}} ) ds\ >  e^{-C} \int dx/\sqrt{1+x^2}$). 
 % \footnote{Alternativemente:
 %$$
 %h=\frac{1}{\left(1+|x|e^y\right)^2}dx^2+dy^2 \geq \left\{
%\begin{array}{lr}
%\frac{1}{\left(1+e^y\right)^2}dx^2+dy^2 & \forall |x| \leq %1 \\
%\frac{1}{\left(1+e^y\right)^2} (d\log |x| )^2+dy^2 & %\forall |x| \geq 1
%\end{array}  
% \right.
% $$  
% and the completeness of $h$ follows easily  taking into account that the metrics on the right are complete (in the manifold with boundary where they are defined) because they are  warped products of complete Riemannian metrics.}
 
{\em Bound with Finslerian $H$}. Proposition~\ref{p_4.2} 
should be applied when the bound with a Riemannian metric $h$ in Propositions \ref{p_4.3} and  \ref{c_4.3}  imply a loss of sharpness. Indeed, modify the example above putting $(M,g_R)= \R^+:=\{x\in\R: x>0\}$, $W_{x}=f(x)\partial_x$, for some smooth function $f$ such that $ 1-x^2  \leq f(x)\leq 2$ on $(0,1]$  and $f$ 
%is bounded $|f|\leq C$ (or at most linear) 
  is 0  on $[2,\infty)$.  	
  %\footnote{he hecho $f$ cero en $[2,+\infty)$ para no tener problemas de completitud}
  The incompleteness of $g_R$ close to the origin makes it impossible to apply Propositions \ref{p_4.3} or \ref{c_4.3}. However, one can obtain the  forward   completeness of the WRS by applying Proposition  \ref{p_4.2} to the non-reversible Finsler metric:
$$
H_x(v)= \left\{
\begin{array}{ll}
 v/3  & \forall v\geq 0 \\
-v/x^2 & \forall v\leq 0
\end{array}
\right.  \qquad \qquad x\in \R^+, \; v\in T_x\R^+ .
$$
The  forward  completeness of $H$ follows because its unique unit  geodesic $[0,L)\ni s \mapsto x(s)$ with $x(0)=1, x'(0)<0$, satisfies  $x'(s)=-x^2(s)$, thus $L=\int_0^Lds=\int_0^1 dx/x^2=\infty$. 
\end{exe}

\begin{rem}  In terms of the SSTK splitting \eqref{e_SSTKsplitting} and using the usual index notation $\omega_i= -(g_0)_{ij}W^j$, $(g_R)_{ij}=(g_0)_{ij}/(\Lambda+\omega_k \, \omega^k)$,  the metrics $h$ and $h^*$ read  
$$\begin{array}{rl} 
h_{ij}= & 
\frac{(g_0)_{ij}}{\Lambda+\omega_k \,  \omega^k} -
\frac{1}{1+\frac{\omega_k \,  \omega^k}{\Lambda+\omega_k \,  \omega^k}}
\frac{\omega_i \, \omega_j}{(\Lambda+\omega_k \,  \omega^k)^2} 
\equiv 
(g_0)_{ij} - 
\frac{\omega_i \, \omega_j}{1+\omega_k \,  \omega^k} 
 \\
h^*_{ij}= & \frac{(g_0)_{ij}}{\left(\sqrt{\Lambda+\omega_k \,  \omega^k}+ \sqrt{\omega_k \,  \omega^k}\right)^2} \equiv \frac{(g_0)_{ij}}{\left(1+ \sqrt{\omega_k \,  \omega^k}\right)^2}
\end{array} $$
where  the indices are raised and lowered with $g_0$ and  the last expression in each equality holds under the choice $\Lambda+\omega_k \,  \omega^k \equiv 1$   in the conformal class of $g$.

\end{rem}

\subsection{Further examples: ergospheres and Killing horizons}\label{s_4.2}
Consider the Lorentzian metric $g$ on  $\R\times \R^3$ in spherical coordinates $(t,r,\theta,\varphi )$,  
\[g= -\Lambda(r) dt^2 + dr^2 + r^2 d\theta^2 + g_{t\theta}(r) (dt d\theta + d\theta dt) 
  + r^2\sin^2\theta d\varphi^2 
\] 
where, for $r\in (1/2,3/2)$, we will  choose $g_{t \theta}(r)=1$ and $\Lambda(r)= (r-1)^m$ for some $m=1, 2, \dots$.  As we will focus on the hypersurface $r=1$, we will assume that the metric matches with Lorentz-Minkowski $\LL^4$ for $|r-1|>2/3$.

First, notice that this is an SSTK spacetime where
$g_0$ is the usual metric of $\R^2$ and $\omega=g_{t\theta}(r)d\theta$. Thus, $|\omega|_0=|g_{t\theta}(r)|/r$, and
$$
W=-\frac{1}{r^2}\partial_\theta , \qquad 
g_R=\frac{r^2}{1+r^2(r-1)^{m}} g_0 \qquad \qquad \hbox{when} \; r\in (1/2,3/2).
$$
Due to the fact that $g_R$ is complete (it agrees with $g_0$ outside the compact subset $|r-1| \leq 2/3$) and  $|W|_{R}$  is bounded, any of the criteria in the previous section implies that the slices of $t$ are Cauchy hypersurfaces.

The (hyper)surface $S$ given as $r=1$  can be seen as  an {\em ergosphere}, because  the sign of $\Lambda$ changes there. This hypersurface is always timelike and so, neither the exterior region $r>1$ nor the interior one $r<1$ are globally hyperbolic. 
One can also check  that the slices of $t$ are not Cauchy hypersurfaces for these regions by using the $\bar F$-separation.  Indeed,  consider the region  $r\in (1-\epsilon, 1+\epsilon)$ for small $\epsilon>0$.  The vector field $Z=(\partial_r -  \partial_\theta)  /2$ lies inside $\Sigma$ because $\Sigma$ is just obtained by taking the usual $g_0$-unit bundle and displacing  it with $-\partial_\theta/r^2$; thus, the $\bar F$-length of $Z$ is smaller than one. As the integral curves of $Z$ must cross $S$, they yield  non-compact $\bar F$-bounded subsets for both, the inner and the outer regions. %\footnote{\br Aqui solo podemos decir con ese razonamiento que los slices no son de Cauchy, a menos que refinemos los resultados, no?\er } 
 Extending our computations (see the next example), it is not difficult to check  also  the lack of global hyperbolicity by using the $\bar F$ distance. 

Even though we have focused on completeness and Cauchy hypersurfaces, other properties of causality can be studied, suggesting that  $\bar F$ can  also be useful beyond our scope in  this article.  A computation shows
\begin{equation}\label{e_acceleration}
\nabla_{\partial_t } \partial_t = \Gamma_{tt}^r\partial_r= -\frac{1}{2} g^{rr}\frac{\partial g_{tt}}{\partial r} \partial_r= \frac{1}{2}  \frac{\partial \Lambda}{\partial r} \partial_r = \frac{m}{2} (r-1)^{m-1} \partial_r , \qquad r\in (1/2,3/2).
\end{equation}
Thus, when $m>1$, the integral curves of $\partial_t$ are 
geodesics, but when $m=1$ they are not. This property is related 
to the light-convexity of $S$ (see \cite{CGS} for background) and, then, to the 
causal simplicity of the regions $r>1$ and $r<1$. Indeed, for 
$m=1$, the inner region $r<1$ cannot be causally simple, as there 
are  lightlike geodesics starting at  this region that touch $S
$ and come back to the inner region (those geodesics of 
the spacetime with  initial velocity parallel to $\partial_t$ on 
$S$)\footnote{To understand this easily,  \eqref{e_acceleration} 
implies that the integral curves of $\partial_t$ are accelerated 
upwards, so, geodesics with initial velocity in $\partial_t$ 
should come from and go into inwards. 
Analytically, if $\rho(s)=(t(s),r(s),\theta(s),\varphi(s))$ is such a 
geodesic, at $s=0$ one has $r'(0)=\theta'(0)=\varphi'(0)=0$, thus, $r''(0)+
\Gamma_{tt}^r (t'(0))^2=0$ and, so, $r''(0)<0$. 
%The reader can 
%obtain further properties of lightlike geodesics (and its projected $\bar F$, $F_l$ -geodesics)  with, for example, constant $\varphi$, 
%by taking into account that, apart from $g(\rho',\rho')=0$, one has the constants of motion $C_t=g(\partial_t,\rho'), C_
%\theta=g(\partial_\theta,\rho')$ which imply $\theta'=
%(g_{t\theta} C_t + \Lambda C_\theta)/(g_{t\theta}%^2+r^2\Lambda)$.
}. Such a property also implies the lack of w-convexity of the inner region (and, in a natural sense, the lack of $\bar F$-convexity of $S$), which characterizes causal convexity in terms of  $\Sigma$.

Finally, consider the following variation of the previous example:
\[g= -\Lambda(r) dt^2 + dr^2  + g_{tr}(r) (dt dr + dr dt) + r^2 d\theta^2 + r^2\sin^2\theta d\varphi^2 
\] 
where, again,  we  choose $\Lambda(r)= (r-1)^m$ for $m=1, 2, \dots$, 
$g_{t r}(r)=1$   when  $r\in (1/2,3/2)$ and  $\LL^4$  when  $|r-1|>2/3$. Now, one has:
$$
W=-\partial_r , \qquad 
g_R=\frac{1}{1+(r-1)^{m}} g_0 \qquad \qquad \hbox{when} \; r\in (1/2,3/2).
$$
As in the previous case, the slices of $t$  are  Cauchy hypersurfaces for the full spacetime.

Now, the surface $S$ given as $r=1$ is a {\em null hypersurface}, as $g$ becomes degenerate there; even more, it  can be regarded as a  Killing horizon.  If one considers only the inner $r<1$ or outer $r>1$ regions, again, the slices $t=$ constant are not Cauchy. 
However, these regions are globally  hyperbolic  (this property goes a bit beyond our previous study,  but  it shows further applications of $\bar F$).   In fact, for, say,  the region $r>1$,  a closer look at the incompleteness of $\bar F$ shows that $\bar F$ is forward incomplete but backward complete.  
 In order to check this,  the relevant curves can be taken as   $(s_-,s_+)\ni s\mapsto (r(s),\theta_0, \varphi_0)$ with $\theta_0, \varphi_0$ constants and $r(0)>1$.  From \eqref{randers-kropina}, the unit curves (necessarily $\bar F$-geodesics)\footnote{Recall that all the direct computations in this example (either for $\bar F$ or for other more classical procedures to study global hyperbolicity) become especially simple,  because the $M$ part of the SSTK spacetime is essentially 1-dimensional, as the  coordinates $\theta, \varphi$ do not play any relevant role.}  
satisfy $1= r'(s)/(-1+\epsilon \sqrt{1+(r-1)^m})$ where $\epsilon=$ sign$(r'(s))\in \{\pm 1\}$  when $r\in (1/2,3/2)$ and they are the Euclidean unit geodesics if $|r-1|>2/3$. Clearly the $\bar F$-geodesics with $\epsilon=1$ (resp. $\epsilon=-1$) are forward (resp. backward) complete. Moreover, the geodesic with $\epsilon=-1$ is also clearly forward incomplete. However, the geodesics with $\epsilon=1$ are backward complete.  Indeed, assuming $r(0)=2$ (as $r'$ cannot vanish one can focus only in one geodesic),  \begin{multline*}
s_-=\int_2^1   dr /(-1+\sqrt{(r-1)^m+1}))\\=  -\int_1^2 (1+\sqrt{(r-1)^m+1}) \, dr/(r-1)^m \\ \leq  -\int_1^2 dr/(r-1)^m = -\infty,\end{multline*} as required. 
 These properties of completenes are    sufficient  for the compactness of the intersections between the forward and backward $\bar F$-balls and, then, for the compactness of the corresponding intersections of the $\Sigma$-balls, the latter property being  a   characterization of global hyperbolicity,  as proven in\footnote{In any case,  the readers used to the stuff in Mathematical Relativity can reason alternatively that these regions are globally hyperbolic because they admit $S$ as a  conformal boundary with no timelike points, which is a known characterization of global hyperbolicity (see \cite[Corollary 4.34]{FHSatmp} for a precise formulation of this result).  } \cite[Th. 5.9]{CJSwind}.   

Being the inner and outer regions globally hyperbolic, they are causally simple too. However, it is interesting to consider again the lightlike  geodesics of the spacetime tangent to\footnote{The fact that they remain in $S$ implies its light-convexity with respect to the  inner and outer regions and, then, the causal simplicity of these regions, see \cite{CGS}.} $S$. First a straightforward computation shows:
\begin{equation}\label{e_acceleration2}
\nabla_{\partial_t } \partial_t = \Gamma_{tt}^t \partial_t + \Gamma_{tt}^r \partial_r  =  \frac{g_{tr}}{2 (\Lambda +g_{tr}^2)} \frac{\partial \Lambda}{\partial r} \partial_t + \frac{\Lambda}{2( \Lambda +g_{tr}^2)}  \frac{\partial \Lambda}{\partial r} \partial_r.
\end{equation}
%\footnote{\br Lo que hay en rojo es igual a 1 en $r=1$, pero o se quita tambien en el denominador de la derecha o se deja asi\er}
So,  one has $\nabla_{\partial_t } \partial_t =  \frac{1}{2} \frac{\partial \Lambda}{\partial r} \partial_t$  on $S$. This means that the integral curves of $\partial_t$  (which  are  the null generators of $S$)  become geodesics when $m>1$ and pregeodesics when $m=1$.  These geodesics plus the ones in the previous example fulfil all  the possible types of lightlike geodesics orthogonal to $K=\partial_t$, described in part (3) of \S \ref{s_313}  (see Lemma 3.21, Th. 6.3(c) in
\cite{CJSwind} for details). 

Summing up, even though the previous examples are very simple and can be handled directly  by means of the explicit computations of lightlike geodesics, causal futures etc., they show the applicability of  both, the general methods introduced in \cite{CJSwind}  and  the additional tools and criteria introduced here,  which are  valid for general SSTK spacetimes with no restrictions on energy conditions, asymptotic behaviors,  etc.

\section*{Acknowledgements}

%The authors warmly acknowledge  ... for helpful conversations on the topics of this
%paper. \sout{ and the anonimous referee for his interesting comments.}
   Partially
supported by Spanish  MINECO/FEDER project reference
MTM2015-65430-P and Fundaci\'on S\'eneca
(Regi\'on de Murcia) project 19901/GERM/15  (MAJ) and MTM2013-47828-C2-1-P (MS).

\end{document}